\documentclass[twoside,leqno,10pt, A4]{amsart}
\usepackage{amsfonts}
\usepackage{amsmath}
\usepackage{amscd}
\usepackage{amssymb}
\usepackage{amsthm}
\usepackage{amsrefs}
\usepackage{latexsym}
\usepackage{mathrsfs}
\usepackage{bbm}
\usepackage{amscd}
\usepackage{amssymb}
\usepackage{amsthm}
\usepackage{amsrefs}
\usepackage{latexsym}
\usepackage{mathrsfs}
\usepackage{bbm}
\usepackage{enumerate}
\usepackage{graphicx}
\usepackage{color}
\begin{document}

\newtheorem{theorem}[subsection]{Theorem}
\newtheorem{proposition}[subsection]{Proposition}
\newtheorem{lemma}[subsection]{Lemma}
\newtheorem{corollary}[subsection]{Corollary}
\newtheorem{conjecture}[subsection]{Conjecture}
\newtheorem{prop}[subsection]{Proposition}
\newtheorem{defin}[subsection]{Definition}

\numberwithin{equation}{section}
\newcommand{\mr}{\ensuremath{\mathbb R}}
\newcommand{\mc}{\ensuremath{\mathbb C}}
\newcommand{\dif}{\mathrm{d}}
\newcommand{\intz}{\mathbb{Z}}
\newcommand{\ratq}{\mathbb{Q}}
\newcommand{\natn}{\mathbb{N}}
\newcommand{\comc}{\mathbb{C}}
\newcommand{\rear}{\mathbb{R}}
\newcommand{\prip}{\mathbb{P}}
\newcommand{\uph}{\mathbb{H}}
\newcommand{\fief}{\mathbb{F}}
\newcommand{\majorarc}{\mathfrak{M}}
\newcommand{\minorarc}{\mathfrak{m}}
\newcommand{\sings}{\mathfrak{S}}
\newcommand{\fA}{\ensuremath{\mathfrak A}}
\newcommand{\mn}{\ensuremath{\mathbb N}}
\newcommand{\mq}{\ensuremath{\mathbb Q}}
\newcommand{\half}{\tfrac{1}{2}}
\newcommand{\f}{f\times \chi}
\newcommand{\summ}{\mathop{{\sum}^{\star}}}
\newcommand{\chiq}{\chi \bmod q}
\newcommand{\chidb}{\chi \bmod db}
\newcommand{\chid}{\chi \bmod d}
\newcommand{\sym}{\text{sym}^2}
\newcommand{\hhalf}{\tfrac{1}{2}}
\newcommand{\sumstar}{\sideset{}{^*}\sum}
\newcommand{\sumprime}{\sideset{}{'}\sum}
\newcommand{\sumprimeprime}{\sideset{}{''}\sum}
\newcommand{\sumflat}{\sideset{}{^\flat}\sum}
\newcommand{\shortmod}{\ensuremath{\negthickspace \negthickspace \negthickspace \pmod}}
\newcommand{\V}{V\left(\frac{nm}{q^2}\right)}
\newcommand{\sumi}{\mathop{{\sum}^{\dagger}}}
\newcommand{\mz}{\ensuremath{\mathbb Z}}
\newcommand{\leg}[2]{\left(\frac{#1}{#2}\right)}
\newcommand{\muK}{\mu_{\omega}}
\newcommand{\thalf}{\tfrac12}
\newcommand{\lp}{\left(}
\newcommand{\rp}{\right)}
\newcommand{\Lam}{\Lambda_{[i]}}
\newcommand{\lam}{\lambda}
\newcommand{\af}{\mathfrak{a}}
\newcommand{\sw}{S_{[i]}(X,Y;\Phi,\Psi)}
\newcommand{\lz}{\left(}
\newcommand{\pz}{\right)}
\newcommand{\bfrac}[2]{\lz\frac{#1}{#2}\pz}
\newcommand{\odd}{\mathrm{\ primary}}
\newcommand{\even}{\text{ even}}
\newcommand{\res}{\mathrm{Res}}
\newcommand{\sumn}{\sumstar_{(c,1+i)=1}  w\left( \frac {N(c)}X \right)}
\newcommand{\lab}{\left|}
\newcommand{\rab}{\right|}
\newcommand{\Go}{\Gamma_{o}}
\newcommand{\Ge}{\Gamma_{e}}
\newcommand{\M}{\widehat}
\def\su#1{\sum_{\substack{#1}}}

\theoremstyle{plain}
\newtheorem{conj}{Conjecture}
\newtheorem{remark}[subsection]{Remark}

\newcommand{\pfrac}[2]{\left(\frac{#1}{#2}\right)}
\newcommand{\pmfrac}[2]{\left(\mfrac{#1}{#2}\right)}
\newcommand{\ptfrac}[2]{\left(\tfrac{#1}{#2}\right)}
\newcommand{\pMatrix}[4]{\left(\begin{matrix}#1 & #2 \\ #3 & #4\end{matrix}\right)}
\newcommand{\ppMatrix}[4]{\left(\!\pMatrix{#1}{#2}{#3}{#4}\!\right)}
\renewcommand{\pmatrix}[4]{\left(\begin{smallmatrix}#1 & #2 \\ #3 & #4\end{smallmatrix}\right)}
\def\en{{\mathbf{\,e}}_n}

\newcommand{\ppmod}[1]{\hspace{-0.15cm}\pmod{#1}}
\newcommand{\ccom}[1]{{\color{red}{Chantal: #1}} }
\newcommand{\acom}[1]{{\color{blue}{Alia: #1}} }
\newcommand{\alexcom}[1]{{\color{green}{Alex: #1}} }
\newcommand{\hcom}[1]{{\color{brown}{Hua: #1}} }

\makeatletter
\def\widebreve{\mathpalette\wide@breve}
\def\wide@breve#1#2{\sbox\z@{$#1#2$}%
     \mathop{\vbox{\m@th\ialign{##\crcr
\kern0.08em\brevefill#1{0.8\wd\z@}\crcr\noalign{\nointerlineskip}%
                    $\hss#1#2\hss$\crcr}}}\limits}
\def\brevefill#1#2{$\m@th\sbox\tw@{$#1($}%
  \hss\resizebox{#2}{\wd\tw@}{\rotatebox[origin=c]{90}{\upshape(}}\hss$}
\makeatletter

\title[Shifted moments of quadratic Dirichlet $L$-functions]{Shifted moments of quadratic Dirichlet $L$-functions}

\author[P. Gao]{Peng Gao}
\address{School of Mathematical Sciences, Beihang University, Beijing 100191, China}
\email{penggao@buaa.edu.cn}

\author[L. Zhao]{Liangyi Zhao}
\address{School of Mathematics and Statistics, University of New South Wales, Sydney NSW 2052, Australia}
\email{l.zhao@unsw.edu.au}

\begin{abstract}
We establish sharp upper bounds for shifted moments of quadratic Dirichlet $L$-function under the generalized Riemann hypothesis.  Our result is then used to
prove bounds for moments of quadratic Dirichlet character sums.
\end{abstract}

\maketitle

\noindent {\bf Mathematics Subject Classification (2010)}: 11L40, 11M06  \newline

\noindent {\bf Keywords}:  quadratic Dirichlet characters, quadratic Dirichlet $L$-functions, shifted moments

\section{Introduction}\label{sec 1}

   Moments of families of $L$-functions have been an important subject of study in number theory for their deep arithmetic implications.  Building on the connections with random matrix theory, conjectural asymptotic formulas, due to J. P. Keating and N. C. Snaith \cite{Keating-Snaith02}, are available for moments of various families of $L$-functions.  More precise predictions on the asymptotic behaviors of these moments with lower order terms are conjectured by J. B. Conrey, D. W. Farmer, J. P. Keating, M. O. Rubinstein and N. C. Snaith \cite{CFKRS}. \newline

    In \cite{Sound2009}, K. Soundararajan developed a method that makes attainable the predicted upper bounds for moments of $L$-functions under the generalized Riemann hypothesis (GRH). This strategy was further refined by A. J. Harper \cite{Harper} in obtaining sharp upper bounds for moments of $L$-functions conditionally. Subsequent works on upper bounds for moments of families of $L$-functions using the above approach of Soundararajan and Harper include \cites{S&Y,QShen,Szab, Sono14, Sono16, Munsch17, Curran, Chandee11, NSW}. Note that, rather than focusing only on moments of $L$-functions at the central point, most of the above-mentioned works investigate shifted moments of $L$-function at points on the critical line.  In fact, as pointed out in \cites{Chandee11, S&Y}, this will allow one to understand the correlation between the values of $L$-functions on the critical line. Moreover, it is shown in \cites{Munsch17, Szab} that these shifted moments can be applied to obtain bounds for moments of character sums. In particular, B. Szab\'o \cite[Theorem 3]{Szab} proved, via the above approach, that for any real $k>2$, any large integer $q$ and $y \in \rear$ with $2 \leq y \leq q^{1/2}$,
\begin{align*}
   \sum_{\chi\in X_q^*}\bigg|\sum_{n\leq y} \chi(n)\bigg|^{2k} \ll_k \varphi(q) y^k(\log y)^{(k-1)^2}.
\end{align*}
Here $X_q^*$ denotes the set of primitive Dirichlet characters modulo $q$ and $\varphi(q)$ Euler's totient function. \newline

   Another interesting application concerns with moments of quadratic Dirichlet character sums. In this case, a conjecture of M. Jutila \cite{Jutila2} asserts that for any positive integer $m$, there are constants $c_1(m)$, $c_2(m)$, with values depending on $m$ only, such that
\begin{align}
\label{genJacobibound}
   \sum_{\substack {\chi \in \mathcal S(X) }} \big| \sum_{n \leq Y} \chi(n) \big|^{2m} \leq c_1(m)XY^m(\log X)^{c_2(m)},
\end{align}
  where $\mathcal S(X)$ stands for the set of all non-principal quadratic Dirichlet characters of modulus at most $X$. \newline

  In \cites{Jutila1}, Julita established \eqref{genJacobibound} for $m=1$ with $c_2(m)=8$. This was later improved to by M. V. Armon \cite[Theorem 2]{Armon} who showed that we can take $c_2(m) = 1$. Other related bounds can be found in \cites{MVa2, Virtanen}. \newline

  In \cite{G&Zhao2024}, the authors confirmed a smoothed version of the above conjecture of Jutila under GRH, applying upper bounds on moments of quadratic Dirichlet $L$-functions. More precisely, we showed that for large $X$, $Y$ and any real $m \geq 1/2$,
\begin{align}
\label{Smoothdef}
  \sumstar_{\substack{d \leq X \\ (d,2)=1}}\Big | \sum_{n}\chi^{(8d)}(n)W \Big(\frac nY \Big )\Big |^{2m} \ll XY^m(\log X)^{m(2m+1)}.
\end{align}
Here $\sum^*$ stands for the sum over square-free integers through out the paper, $W$ is any non-negative, smooth function compactly supported on the set of positive real numbers and $\chi^{(8d)}$ denotes the Jacobi symbol $\leg {8d}{\cdot}$. Note that (see \cite{sound1}) the character  $\chi^{(8d)}$ is primitive modulo $8d$ for any positive, odd and square-free $d$. \newline

The aim of this paper is to prove an unsmoothed version of \eqref{Smoothdef}. To this end, we first apply the above-mentioned approach of Soundararajan \cite{Sound2009} with its refinement by Harper \cite{Harper} to extend the results in \cites{Harper, QShen}.  This establishes the following sharp upper bounds on moments of shifted quadratic Dirichlet $L$-functions.
\begin{theorem}
\label{t1}
 With the notation as above and the truth of GRH, let $k\geq 1$ be a fixed integer and $a_1,\ldots, a_{k}$, $A$ fixed positive real numbers. Suppose that $X$ is a large real number and $t=(t_1,\ldots ,t_{k})$ a real $k$-tuple with $|t_j|\leq  X^A$. Then
\begin{align*}
\begin{split}
 & \sumstar_{\substack{(d,2)=1 \\ d \leq X}}\big| L\big(1/2+it_1,\chi^{(8d)} \big) \big|^{a_1} \cdots \big| L\big(1/2+it_{k},\chi^{(8d)}  \big) \big|^{a_{k}} \\
\ll & X(\log X)^{(a_1^2+\cdots +a_{k}^2)/4} \\
& \times \prod_{1\leq j<l \leq k} \Big|\zeta \Big(1+i(t_j-t_l)+\frac 1{\log X} \Big) \Big|^{a_ja_l/2}\Big|\zeta \Big(1+i(t_j+t_l)+\frac 1{\log X} \Big) \Big|^{a_ja_l/2}\prod_{1\leq j\leq k} \Big|\zeta \Big(1+2it_j+\frac 1{\log X} \Big) \Big|^{a^2_j/4+a_j/2},
\end{split}
\end{align*}
 where $\zeta(s)$ is the Riemann zeta function. Here the implied constant depends on $k$, $A$ and the $a_j$'s, but not on $X$ or the $t_j$'s.
\end{theorem}

  In order to apply Theorem \ref{t1} in practice, one often needs to estimate the Riemann zeta function involved there. For this purpose, we apply \eqref{mertenstype} in Theorem \ref{t1} to readily deduce the following restatement of it.
\begin{corollary}
\label{cor1}
 With the notation as above and the truth of GRH, let $k\geq 1$ be a fixed integer and $a_1,\ldots, a_{k}$, $A$ fixed positive real numbers. Suppose that $X$ is a large real number and $t=(t_1,\ldots ,t_{k})$ a real $k$-tuple with $|t_j|\leq  X^A$. Then
\begin{align*}
\begin{split}
  \sumstar_{\substack{(d,2)=1 \\ d \leq X}} & \big| L\big( \tfrac{1}{2} +it_1,\chi^{(8d)} \big) \big|^{a_1} \cdots \big| L\big(\tfrac{1}{2}+it_{k},\chi^{(8d)}  \big) \big|^{a_{k}} \\
& \ll  X(\log X)^{(a_1^2+\cdots +a_{k}^2)/4} \prod_{1\leq j<l\leq k} g(|t_j-t_l|)^{a_ja_l/2}g(|t_j+t_l|)^{a_ja_l/2}\prod_{1\leq j\leq k} g(|2t_j|)^{a^2_j/4+a_j/2},
\end{split}
\end{align*}
where $g:\mathbb{R}_{\geq 0} \rightarrow \mathbb{R}$ is defined by
\begin{equation} \label{gDef}
g(x) =\begin{cases}
\log X,  & \text{if } x\leq 1/\log X \text{ or } x \geq e^X, \\
1/x, & \text{if }   1/\log X \leq x\leq 10, \\
\log \log x, & \text{if }  10 \leq x \leq e^{X}.
\end{cases}
\end{equation}
Here the implied constant depends on $k$, $A$ and the $a_j$'s, but not on $X$ or the $t_j$'s.
\end{corollary}

As an application of Corollary \ref{cor1}, our next result lends further credence to Jutila's conjecture.
\begin{theorem}
\label{quadraticmean}
With the notation as above and the truth of GRH, for any integer $k \geq 1$ and any real number $m$ satisfying $2m \geq k+1$, we have for large $X$, $Y$  and any $\varepsilon>0$, 
\begin{align}
\label{mainestimation}
 S_m(X,Y):=\sumstar_{\substack{d \leq X \\ (d,2)=1}}\Big | \sum_{n \leq Y}\chi^{(8d)}(n)\Big |^{2m} \ll XY^m(\log X)^{E(m,k,\varepsilon)},
\end{align}
where
\begin{equation} \label{Edef}
 E(m,k,\varepsilon) = \max (2m^2-m+1, (2m-k)^2/4+2m+1+\varepsilon, 2m^2-2mk+3k^2/4+m-3k/4+\varepsilon) .
 \end{equation}
  In particular, we have for $m > (\sqrt{5}+1)/2$, 
\begin{align}
\label{mainestimationmlarge}
 S_m(X,Y) \ll XY^m(\log X)^{2m^2-m+1}.
\end{align}
\end{theorem}

The bound in \eqref{mainestimationmlarge} follows from \eqref{mainestimation} by setting $k=2$ there and noting that $E(m,k,\varepsilon) =2m^2-m+1$ for $m^2-m-1>0$.  Moreover, H\"older's inequality yields that for any real number $n>1$, 
\begin{align*}
 S_m(X,Y) \ll  X^{1-1/n}(S_{mn}(X,Y))^{1/n}. 
\end{align*}
   The above together with Theorem \ref{quadraticmean} then implies that $S_m(X,Y) \ll XY^m(\log X)^{O(1)}$ for any $m>0$, upon choosing $n$ sufficiently enough. 

\section{Preliminaries}
\label{sec 2}

In this section, we cite some results needed in the proof of our theorems.

\subsection{Sums over primes}
\label{sec2.1}

We reserve the letter $p$ for a prime number in this paper.  We have the following result concerning sums over primes.
\begin{lemma}
\label{RS} Let $x \geq 2$ and $\alpha \geq 0$. We have, for some constant $b_1$,
\begin{align}
\label{merten}
\sum_{p\le x} \frac{1}{p} =& \log \log x + b_1+ O\Big(\frac{1}{\log x}\Big), \\
\label{mertenpartialsummation}
\sum_{p\le x} \frac {\log p}{p} =& \log x + O(1),  \quad \mbox{and} \\
\label{mertenstype}
  \sum_{p\leq x} \frac{\cos(\alpha \log p) }{p}=& \log |\zeta(1+1/\log x+i\alpha)| +O(1)
  \leq
\begin{cases}
\log\log x+O(1)            & \text{if }  \alpha\leq 1/\log x \text{ or } \alpha\geq e^x  ,   \\
\log(1/\alpha)+O(1)        & \text{if }  1/\log x\leq \alpha \leq 10,   \\
\log\log\log \alpha + O(1) & \text{if }   10 \leq \alpha \leq e^x.  
\end{cases}
\end{align}
The estimates in the first two cases of \eqref{mertenstype} are unconditional and the third holds under the Riemann hypothesis.
\end{lemma}
\begin{proof}
  The expressions in \eqref{merten} and \eqref{mertenpartialsummation} can be found in parts (d) and (b) of [6, Theorem 2.7], respectively. The equality in  \eqref{mertenstype} is a special case given in \cite[Lemma 3.2]{Kou} and the other estimations can be found in \cite[Lemma 2]{Szab}.
\end{proof}

\subsection{Smoothed character sums}

  We now define for any integer $d$,
\begin{align}
\label{Addef}
\begin{split}
  A(d)=\prod_{\substack{p|d }}\Big(1-\frac {1}{2p}\Big).
\end{split}
\end{align}
As per convention, the empty product is $1$ and the empty sum $0$.  We observe that $A(d)>0$ for any $d$. \newline

  We write $\square$ for a perfect square of rational integers. Our next result is for smoothed quadratic character sums.
\begin{lemma}
\label{prsum}
Suppose that $\Phi$ is a non-negative smooth function whose support is a compact subset of the positive real numbers.  With the notation as above and $k\in \rear$ with $k>0$, for any positive even integer $n$, 
\begin{align}
\label{sumAdnven}
\sumstar\limits_{(d,2)=1}A^{-k}(d)\chi^{(8d)}(n)\Phi \Big( \frac{d}{X} \Big)=\sumstar\limits_{(d,2)=1}\chi^{(8d)}(n)\Phi \Big( \frac{d}{X} \Big)=0.
\end{align}
If $n$ is an odd positive integer, then
\begin{align}
\label{sumAd}
\sumstar_{(d,2)=1} & A(d)^{-k}\chi^{(8d)}(n) \Phi \Big( \frac{d}{X} \Big) \\
=& \delta_{n=\square}{\widehat\Phi}(1) \frac{X}{2}\prod_{(p,2)=1}\Big(1-\frac{1}{p}\Big) \Big(1+\frac{A(p)^{-k}}{p} \Big)\times \prod_{p|n}\Big(1+\frac{A(p)^{-k}}{p} \Big)^{-1}+
O_k(X^{1/2+\varepsilon}n^{1/4+\varepsilon}), \nonumber \\
\label{meancharsum}
\sumstar_{\substack{(d,2)=1}}&  \chi^{(8d)}(n) \Phi\Big(\frac{d}{X}\Big)=
\displaystyle \delta_{n=\square}{\widehat \Phi}(1) \frac{2X}{3\zeta(2)} \prod_{p|n} \Big(\frac p{p+1}\Big) + O (X^{1/2+\varepsilon}n^{1/4+\varepsilon}),
\end{align}
  where $\delta_{n=\square}=1$ if $n$ is a square and $\delta_{n=\square}=0$ otherwise.
\end{lemma}
\begin{proof}
  The relations in \eqref{sumAdnven} are valid trivially, as $\chi^{(8d)}(n)=0$ if $n$ is even.  Moreover, \eqref{sumAd} and \eqref{meancharsum} can be established obtained in a manner similar way.  Thus we shall only prove \eqref{sumAd} here.  Let $\mu(n)$ denote the M\"obius function.  The Mellin inversion gives that the left-hand side of \eqref{sumAd} is
\begin{align}
\label{sumdint}
\sum_{(d,2)=1}A(d)^{-k}\mu^2(d)\chi^{(8d)}(n)\Phi\Big( \frac{d}{X} \Big) 
= \frac 1{2\pi i } \int\limits_{(2)} \chi^{(8)}(n)\Big (\sum_{(d,2)=1}\frac {A(d)^{-k}\mu^2(d)\chi^{(d)}(n)}{d^s}\Big )\widehat{\Phi}(s) X^s \dif s.
\end{align}
Recall that the Mellin transform $\widehat f(s)$ with $s \in \mc$ for any function $f$ is defined by
\begin{equation*}
\widehat f (s) = \int\limits_{0}^{\infty} f(x)x^{s}\frac {\dif x}{x}.
\end{equation*}
and that repeated integration by parts yields that for any integer $E \geq 0$,
\begin{align}
\label{phihatbound}
 \widehat \Phi (s)  \ll  \frac{1}{(1+|s|)^{E}}.
\end{align}
Now we transform the integrand in \eqref{sumdint} via the formula,
\begin{align*}
 \sum_{(d,2)=1}\frac {A(d)^{-k}\mu^2(d)\chi^{(d)}(n)}{d^s}=\prod_{(p,2)=1} \Big( 1+\frac {A(p)^{-k}\leg {p}{n}}{p^s} \Big)=L^{(2)}\Big( s, \leg {\cdot}{n} \Big)P_n(s),
\end{align*}
  where $L^{(n)}(s,\chi)$ denotes the Euler product of $L(s,\chi)$ with the primes dividing $n$ removed and 
\begin{align*}
 P_n(s)=\prod_{(p,2)=1} \Big( 1-\leg {p}{n}p^{-s} \Big) \Big(1+\frac {A(p)^{-k}\leg {p}{n}}{p^s} \Big).
\end{align*}
  Note that $P_n(s)$ is analytic in the region $\Re(s)>\frac 12$. \newline

  We now shift the line of integration in \eqref{sumdint} to the line $\Re(s)=1/2+\varepsilon$, pasing a simple pole at $s=1$ for
 $n=\square$.  The residue is
\begin{align*}
{\widehat\Phi}(1) \frac{X}{2}\prod_{(p,2)=1}\Big(1-\frac{1}{p}\Big) \Big(1+\frac{A(p)^{-k}}{p} \Big)\times \prod_{p|n}\Big(1+\frac{A(p)^{-k}}{p} \Big)^{-1}.
\end{align*}

The convexity bound (see \cite[exercise 3, p. 100]{iwakow}) asserts that for $0 \leq \Re(s) \leq 1$,
\begin{align*}
L\Big( s, \leg {\cdot}{n} \Big) \ll (n\cdot (1+|s|))^{(1-\Re(s))/2+\varepsilon}.
\end{align*}
Now, this, together with \eqref{phihatbound}, implies that the contribution from the integration on the new line can be absorbed in the $O$-term in \eqref{sumAd}. This establishes the lemma.
\end{proof}

\subsection{Upper bounds for quadratic Dirichlet $L$-functions}
\label{sec2.4}

  Let $\Lambda(n)$ denote the von Mangoldt function. We cite the following result for $\log |L(\sigma+it,\chi^{(8d)})|$ from \cite[Proposition 2.3]{Munsch17}.
\begin{lemma}
\label{up}
 Assume the truth of GRH. Let $\sigma \geq 1/2$ and $t$ be real numbers and define $\log^+ t=\max\{ 0,\log t\}$. For any primitive Dirichlet character $\chi$ modulo $q$, we have for any $x \geq 2$,
\begin{equation}
    \label{variant}
    \log |L(\sigma+it, \chi)| \leq \Re \sum_{n\leq x}\frac{\chi(n)\Lambda(n)}{n^{1/2+\max(\sigma-1/2, 1/\log x) +it} \log n}\frac{\log x/n}{\log x}+\frac{\log q+\log^+ t}{\log x}+O\Big( \frac{1}{\log x} \Big).
\end{equation}
\end{lemma}
\begin{proof}
    If $\sigma \geq 1/2+1/\log x$, then we apply \cite[(2.8)]{Munsch17} by setting $s_0=\sigma+it$ there and note that the contribution from  terms involving $F_{\chi}(s_0)$ is negative.  This leads the exponent of $\sigma +it$ on $n$ in the denominator of the summands on the right-hand side of \eqref{variant}.  If $\sigma < 1/2+1/\log x$, then we set $\sigma_0=1/2+1/\log x$ in \cite[Proposition 2.3]{Munsch17} to obtain the exponent of $1/2+1/\log x$ for the afore-mentioned $n$.   Thus the proof is complete upon combining the bounds.
\end{proof}

  We apply the above lemma and argue as in the proof of \cite[Lemma 3]{Szab} to deduce the following bound for sums of $\log |L(\sigma+it,\chi^{(8d)})|$ over various $t$.
\begin{lemma}
\label{lemmasum}
  Let $k$ be a positive integer and let $Q,a_1,a_2,\ldots, a_{k}$ be fixed positive real constants, $x\geq 2$.  Set $a:=a_1+\cdots+ a_{k}$.  Suppose $X$ is a large real number, $d$ any positive, odd and square-free number with $d \leq X$, $\sigma \geq 1/2$ and $t_1,\ldots, t_{k}$ be fixed real numbers with $|t_i|\leq X^Q$. For any integer $n$, let
$$h(n)=:\frac{1}{2}\Re \Big( \sum^{k}_{m=1}a_mn^{-it_m} \Big).$$
Then, we have, under GRH,
\begin{align}
\label{mainupper}
\begin{split}
 \sum^{k}_{m=1} & a_m\log |L(\sigma+it_m,\chi^{(8d)})| \\
    \leq & 2 \sum_{p\leq x} \frac{h(p)\chi^{(8d)}(p)}{p^{1/2+\max(\sigma-1/2, 1/\log x)}}\frac{\log x/p}{\log x}
    +\sum_{p\leq x^{1/2}} \frac{h(p^2)\chi^{(8d)}(p^2)}{p^{1+2\max(\sigma-1/2, 1/\log x)}}+(Q+1)a\frac{\log X}{\log x}+O(1).
\end{split}
\end{align}
\end{lemma}

   We note that
\begin{align}
\label{sump}
\begin{split}
 \sum_{p\leq x^{1/2}} \frac{h(p^2)\chi^{(8d)}(p^2)}{p^{1+2\max(\sigma-1/2, 1/\log x)}}=& \sum_{p\leq x^{1/2}} \frac{h(p^2)}{p^{1+2\max(\sigma-1/2, 1/\log x)}}
 -\sum_{\substack{p\leq x^{1/2} \\ p|d}} \frac{h(p^2)}{p^{1+2\max(\sigma-1/2, 1/\log x)}} \\
 \leq & \sum_{p\leq x^{1/2}} \frac{h(p^2)}{p^{1+2\max(\sigma-1/2, 1/\log x)}}+a\sum_{\substack{p|d}} \frac{1}{2p}.
\end{split}
\end{align}
  Using the inequality $x \leq -\log(1-x)$ for any $0<x<1$, we see that the last expression above is
\begin{align}
\label{sump1}
\begin{split}
 \leq \sum_{p\leq x^{1/2}} \frac{h(p^2)}{p^{1+2\max(\sigma-1/2, 1/\log x)}}-a\sum_{\substack{p|d}} \log (1- \frac{1}{2p}).
\end{split}
\end{align}

Upon applying \eqref{Addef}, \eqref{sump}--\eqref{sump1} in \eqref{mainupper}, we immediately deduce the following simplified version of Lemma \ref{lemmasum}.
\begin{proposition}
\label{lem: logLbound}
Keep the notations of Lemma \ref{lemmasum} and assume the truth of GRH, we have
\begin{align}
\label{logLupperbound}
\begin{split}
\sum^{k}_{m=1} & a_m\log |A(d)L(\sigma+it_m,\chi^{(8d)})| \\
    \leq & 2 \sum_{p\leq x} \frac{h(p)\chi^{(8d)}(p)}{p^{1/2+\max(\sigma-1/2, 1/\log x)}}\frac{\log x/p}{\log x}+
    \sum_{p\leq x^{1/2}} \frac{h(p^2)}{p^{1+2\max(\sigma-1/2, 1/\log x)}}+(Q+1)a\frac{\log X}{\log x}+O(1).
\end{split}
\end{align}
\end{proposition}

   We also note the following upper bounds on moments of quadratic $L$-functions, which can be obtained by modifying the proof of \cite[Theorem 2]{Harper}.
\begin{lemma}
\label{prop: upperbound}
With the notation as above and the truth of GRH, let $k\geq 1$ be a fixed integer and ${\bf a}=(a_1,\ldots ,a_{k}),\ t=(t_1,\ldots ,t_{k})$
 be real $k$-tuples such that $a_i \geq 0$ for all $i$.  Set $a=a_1+\cdots+ a_{k}$.  Then for large real number $X$ and $\sigma \geq 1/2$,
\begin{align*}
   \sumstar_{\substack{(d,2)=1 \\ d \leq X}}\big| L\big(\sigma+it_1,\chi^{(8d)} \big) \big|^{a_1} \cdots \big| L\big(\sigma+it_{k},\chi^{(8d)}  \big) \big|^{a_{k}} \ll_{{\bf a}} &  X(\log X)^{a(a+1)/2}.
\end{align*}
\end{lemma}

\section{Proof of Theorem \ref{t1}}

We remark here that throughout our proof, the implied constants involved in various estimations using $\ll$ or the big-$O$ notations depend on ${\bf a}:=(a_1, \cdots, a_{k})$ only and are uniform with respect to $X$.  We set $a=\sum^k_{i=1}a_i$ and recall the convention that an empty product equals to $1$. \newline

Let $\Phi$ be a smooth, non-negative function such that $\Phi(x) \leq 1$ for all $x$ and that $\Phi$ is supported on $[1/4,3/2]$ satisfying $\Phi(x) =1$ for $x\in [1/2,1]$. Upon dividing $0<d \leq X$ into dyadic blocks, we see that in order to
 prove Theorem \ref{t1}, it suffices to show that for $\sigma=1/2$,
\begin{align}
\label{Lprodboundssmoothed}
\begin{split}
  \sumstar_{(d,2)=1} & \big| L\big(\sigma+it_1,\chi^{(8d)} \big) \big|^{a_1} \cdots \big| L\big(\sigma+it_{k},\chi^{(8d)}  \big) \big|^{a_{k}}\Phi \Big( \frac d{X} \Big) \\
\ll & X(\log X)^{(a_1^2+\cdots +a_{k}^2)/4} \\
& \times \prod_{1\leq j<l \leq k} \big|\zeta(1+i(t_j-t_l)+\frac 1{\log X}) \big|^{a_ja_l/2}\big|\zeta(1+i(t_j+t_l)+\frac 1{\log X}) \big|^{a_ja_l/2}\prod_{1\leq j\leq k} \big|\zeta(1+2it_j+\frac 1{\log X}) \big|^{a^2_j/4+a_j/2}.
\end{split}
\end{align}

We shall treat $\sigma$ as a fixed real number with $\sigma \geq 1/2$ instead of focusing only on the case of $\sigma=1/2$ in the most part of our proof in what follows since we would like to present a proof that is valid for a general $\sigma$.  Of course, Theorem \ref{t1} follows from \eqref{Lprodboundssmoothed} by setting $\sigma = 1/2$. \newline

 Following the ideas of A. J. Harper in \cite{Harper},  we define for a large number $M$, depending on ${\bf a}$ only,
\begin{align*}
\begin{split}
 \alpha_{0} = 0, \;\;\;\;\; \alpha_{j} = \frac{20^{j-1}}{(\log\log X)^{2}} \;\;\; \mbox{for all} \; j \geq 1, \quad
\mathcal{J} = \mathcal{J}_{{\bf a},X} = 1 + \max\{j : \alpha_{j} \leq 10^{-M} \} .
\end{split}
\end{align*}

Set $P_j=(X^{\alpha_{j-1}}, X^{\alpha_{j}}]$ for $1 \leq j \leq \mathcal{J}$. Lemma \ref{RS} gives that for $X$ large enough,
\begin{align}
\label{sumpj}
\begin{split}
 \sum_{p  \in P_{1}} \frac{1}{p} \leq & \log\log X=\alpha^{-1/2}_1, \; \mathcal{J}-j \leq  \frac{\log(1/\alpha_{j})}{\log 20} \quad \mbox{and}   \\
 \sum_{ p \in P_{j+1}} \frac{1}{p}
 =& \log \alpha_{j+1} - \log \alpha_{j} + o(1) =  \log 20 + o(1) \leq 10, \quad 1 \leq j \leq \mathcal{J}-1. 
\end{split}
\end{align}

Denote $\lceil x \rceil = \min \{ n \in \intz : n \geq x\}$ for any $x \in \rear$ and we define a sequence of even natural
  numbers $\{ \ell_j \}_{1 \leq j \leq \mathcal{J}}$ so that $\ell_j =2\lceil e^{B}\alpha^{-3/4}_j \rceil$ with $B$ being a large number depending on ${\bf a}$ only.  For any $x\in \rear$ and any non-negative integer $\ell$, we set
\begin{equation*}
E_{\ell}(x) = \sum_{j=0}^{\ell} \frac{x^{j}}{j!}.
\end{equation*}
Next, we define three functions $h(n, \sigma, x), h_1(n, \sigma, x)$ and $s(n, x)$, totally multiplicative in $n$, with their values at a prime $p$ given by
\begin{align} \label{hbound}
  h(p, \sigma, x)= \frac{2h(p)}{ap^{\max(\sigma-1/2, 1/\log x)}}, \quad h_1(p, \sigma, x)= \frac{4h(p^2)}{a^2p^{2\max(\sigma-1/2, 1/\log x)}}\quad   s(p,x)= \frac{\log (x/p)}{\log x}.
\end{align}

Note that for any integer $n \geq 1$,
\begin{align*}
\begin{split}
  |h(n, \sigma, x)| \leq 1.
\end{split}
\end{align*}

  Let $w(n)$ denote the multiplicative function such that $w(p^{\alpha}) = \alpha!$ for prime powers $p^{\alpha}$.  Setting $x=X^{\alpha_j}$ in \eqref{logLupperbound} renders that
\begin{align}
\label{basicest}
\begin{split}
 & \sum^{k}_{m=1}a_m\log |A(d)L(\sigma+it_m,\chi^{(8d)})| \le a\sum^{j}_{l=1} {\mathcal M}_{l,j}(d)+\frac {a^2}{4}\sum_{p\leq X^{\alpha_j/2}}
 \frac{h_1(p,\sigma, X^{\alpha_j})}{p}+(Q+1)a\alpha^{-1}_j+O(1),
\end{split}
\end{align}
  where
\[ {\mathcal M}_{l,j}(d) = \sum_{p\in P_l}\frac {h(p,\sigma, X^{\alpha_j})\chi^{(8d)}(p)}{\sqrt{p}}s(p, X^{\alpha_j}), \quad 1\leq l \leq j \leq \mathcal{J}. \]

  We also define the following sets:
\begin{align*}
  \mathcal{S}(0) =& \{ (d,2)=1 : |a{\mathcal M}_{1,l}(d)| > \frac {\ell_{1}}{10^3} \; \text{ for some } 1 \leq l \leq \mathcal{J} \} ,   \\
 \mathcal{S}(j) =& \{ (d,2)=1  : |a{\mathcal M}_{m,l}(d)| \leq
 \frac {\ell_{m}}{10^3},  \; \mbox{for all} \; 1 \leq m \leq j \; \mbox{and} \; m \leq l \leq \mathcal{J}, \\
 & \;\;\;\;\; \text{but }  |a{\mathcal M}_{j+1,l}(d)| > \frac {\ell_{j+1}}{10^3} \; \text{ for some } j+1 \leq l \leq \mathcal{J} \} ,  \quad  1\leq j \leq \mathcal{J}, \\
 \mathcal{S}(\mathcal{J}) =& \{(d,2)=1  : |a{\mathcal M}_{m,
\mathcal{J}}(d)| \leq \frac {\ell_{m}}{10^3} \; \forall 1 \leq m \leq \mathcal{J}\}.
\end{align*}

  Now we note that
\begin{align}
\label{S0est}
\begin{split}
 \sumstar_{\substack{(d,2)=1 \\ d \in \mathcal{S}(0)}}\Phi \Big( \frac d{X} \Big)   \leq &  \sumstar_{\substack{(d,2)=1}} \sum^{\mathcal{J}}_{l=1}
\Big ( \frac {10^3}{\ell_1}{|a\mathcal
M}_{1, l}(d)| \Big)^{2\lceil 1/(10^3\alpha_{1})\rceil }\Phi \Big( \frac d{X} \Big).
\end{split}
\end{align}

 Let $\Omega(n)$ be the number of prime powers dividing $n$.  These notations lead to
\begin{align}
\label{M1fbound}
\begin{split}
 \sumstar_{\substack{(d,2)=1}} & \sum^{\mathcal{J}}_{l=1}
\Big ( \frac {10^3}{\ell_1} {|a\mathcal
M}_{1, l}(d)| \Big)^{2\lceil 1/(10^3\alpha_{1})\rceil }\Phi \Big( \frac d{X} \Big) \\
= & \sum^{\mathcal{J}}_{l=1}
 \Big ( \frac {a \cdot 10^3}{\ell_1} \Big)^{2\lceil 1/(10^3\alpha_{1})\rceil }  \sum_{ \substack{ n \\ \Omega(n) = 2\lceil 1/(10^3\alpha_{1})\rceil \\ p|n \implies
p \in P_1}}
\frac{(2\lceil 1/(10^3\alpha_{1})\rceil  )!s(n, X^{\alpha_l})}{\sqrt{n}}\frac{h(n,\sigma, X^{\alpha_l})
  }{w(n)} \sumstar_{\substack{(d,2)=1}}  \chi^{(8d)}(n) \Phi \Big( \frac d{X} \Big).
\end{split}
\end{align}

   We apply Lemma \ref{prsum} to evaluate the inner-most sum over $d$ on the right-hand side of \eqref{M1fbound} and the observation that $|s(n,  X^{\alpha_l})| \leq 1$ for all $n$ appearing in the sum.  The contribution from the $O$-term in \eqref{meancharsum} is
\begin{align}
\label{S0bounderror}
\begin{split}
  \ll &  X^{1/2+\varepsilon}\mathcal{J} \Big ( \frac {a \cdot 10^3}{\ell_1} \Big)^{2\lceil 1/(10^3\alpha_{1})\rceil } \max_{1 \leq l \leq \mathcal{J}} \sum_{ \substack{ n \\ \Omega(n) = 2\lceil 1/(10^3\alpha_{1})\rceil \\ p|n \implies
p \in P_1}}
\frac{(2\lceil 1/(10^3\alpha_{1})\rceil  )!s(n,X^{\alpha_l})}{\sqrt{n}}\frac{|h(n,\sigma, X^{\alpha_l})| }{w(n)}n^{1/4+\varepsilon} \\
\ll &   X^{1/2+\varepsilon}\mathcal{J}\Big ( \frac {a \cdot  10^3}{\ell_1} \Big)^{2\lceil 1/(10^3\alpha_{1})\rceil } \Big(\sum_{p \in P_1} \frac{1}{p^{1/4-\varepsilon}} \Big )^{2\lceil 1/(10^3\alpha_{1})\rceil} \ll X e^{-(\log\log X)^{2}/20} .
\end{split}
\end{align}

Moreover, the contribution from the main term in \eqref{meancharsum}  is
\begin{align}
\label{S00bound}
\begin{split}
  \ll & X \mathcal{J} \Big ( \frac {a \cdot  10^3}{\ell_1} \Big)^{2\lceil 1/(10^3\alpha_{1})\rceil } \max_{1 \leq l \leq \mathcal{J}}\sum_{ \substack{ n=\square \\ \Omega(n) = 2\lceil 1/(10^3\alpha_{1})\rceil \\ p|n \implies
p \in P_1}}
\frac{(2\lceil 1/(10^3\alpha_{1})\rceil  )!s(n,X^{\alpha_l})}{\sqrt{n}}\frac{h(n,\sigma, X^{\alpha_l})
  }{w(n)}\prod_{p|n} \Big (\frac {p}{p+1}\Big ).
\end{split}
\end{align}

  Upon replacing $n$ by $n^2$ and noting that $s(n^2,X^{\alpha_l})=s(n,X^{\alpha_l})^2$, $h(n^2,\sigma, X^{\alpha_l})=h(n,\sigma, X^{\alpha_l})^2 \geq 0$, $1/w(n^2) \leq 1/w(n)$, and $p/(p+1) \leq 1$, we deduce that \eqref{S00bound} is
\begin{align}
\label{S0bound1}
\begin{split}
  \ll & X \mathcal{J} \Big ( \frac {a \cdot  10^3}{\ell_1} \Big)^{2\lceil 1/(10^3\alpha_{1})\rceil } \max_{1 \leq l \leq \mathcal{J}} \sum_{ \substack{ n \\ \Omega(n) = \lceil 1/(10^3\alpha_{1})\rceil \\ p|n \implies
p \in P_1}}
\frac{(2\lceil 1/(10^3\alpha_{1})\rceil  )!s(n,X^{\alpha_l})^2}{n}\frac{h(n,\sigma, X^{\alpha_l})^2
  }{w(n)}\\
\ll &  X\mathcal{J} \Big ( \frac {a \cdot 10^3}{\ell_1} \Big)^{2\lceil 1/(10^3\alpha_{1})\rceil }
\frac {(2\lceil 1/(10^3\alpha_{1})\rceil  )!}{\lceil 1/(10^3\alpha_{1})\rceil !}\max_{1 \leq l \leq \mathcal{J}} \Big (\sum_{p \in P_1 }
\frac {h(p,\sigma, X^{\alpha_l})^2 s(p,X^{\alpha_l})^2}{p}\Big )^{\lceil 1/(10^3\alpha_{1})\rceil}.
\end{split}
\end{align}

Now Stirling's formula (see \cite[(5.112)]{iwakow}) implies that
\begin{align}
\label{Stirling}
\begin{split}
 \Big( \frac me \Big)^m \leq m! \leq \sqrt{m} \Big( \frac {m }{e} \Big)^{m}.
\end{split}
\end{align}

Thus, \eqref{Stirling}, together with the observation that $s(p,X^{\alpha_l})\leq 1$, and \eqref{hbound} imply that the last expression in \eqref{S0bound1} is
\begin{align}
\label{S0bound2}
\begin{split}
\ll &  X\mathcal{J} \lceil 1/(10^3\alpha_{1})\rceil \Big ( \frac {a \cdot 10^3}{\ell_1} \Big)^{2\lceil 1/(10^3\alpha_{1})\rceil }\Big(\frac {(4\lceil 1/(10^3\alpha_{1})\rceil  )}{e}\Big )^{\lceil 1/(10^3\alpha_{1})\rceil}\Big (\sum_{p \leq X^{\alpha_1} }\frac {1}{p}\Big )^{\lceil 1/(10^3\alpha_{1})\rceil}.
\end{split}
\end{align}

Lemma \ref{RS} and \eqref{sumpj} yield that, upon taking $B$ large enough, \eqref{S0bound2} is
 \begin{align}
\label{S0upperbound}
\begin{split}
 \ll  Xe^{-\alpha_1^{-1}/20}=Xe^{-(\log\log X)^{2}/20}  .
\end{split}
\end{align}

  We conclude from \eqref{S0est}, \eqref{S0bounderror} and \eqref{S0upperbound} that
\begin{align}
\label{S0est1}
\begin{split}
 \sumstar_{\substack{(d,2)=1 \\ d \in \mathcal{S}(0)}}\Phi \Big( \frac d{X} \Big)  \ll Xe^{-(\log\log X)^{2}/20}.
\end{split}
\end{align}

Via H\"older's inequality, we arrive at
\begin{align}
\label{LS0bound}
\begin{split}
  \sumstar_{\substack{(d,2)=1 \\d \in S(0) }} & \big| L\big(\sigma+it_1,\chi^{(8d)} \big) \big|^{a_1} \cdots \big| L\big(\sigma+it_{k},\chi^{(8d)}  \big) \big|^{a_{k}} \Phi \Big( \frac d{X} \Big) \\
\leq &  \Big ( \sumstar_{\substack{(d,2)=1 \\d \in S(0) }} \Phi \Big( \frac d{X} \Big)  \Big )^{1/2} \Big (
 \sumstar_{\substack{(d,2)=1 }}\big| L\big(\sigma+it_1,\chi^{(8d)} \big) \big|^{2a_1} \cdots \big| L\big(\sigma+it_{k},\chi^{(8d)}  \big) \big|^{2a_{k}} \Phi \Big( \frac d{X} \Big) \Big)^{1/2}.
\end{split}
\end{align}

  Note that by \cite[Theorem 6.7]{MVa1}, we have for $|t_j|, |t_l|\leq X^A$ and $X$ large enough, 
\begin{align*}
 \big|\zeta(1+i(t_j-t_l)+\frac 1{\log X}) \big| \gg \min (\frac 1{\log X}, \frac {1}{\log (|t_j-t_l|+4)}) \gg \frac 1{\log X}.
\end{align*}  
  
  It then follows from Lemma \ref{prop: upperbound}, \eqref{S0est1} and \eqref{LS0bound} that
\begin{align*}
\sumstar_{\substack{(d,2)=1 \\d \in S(0) }} & \big| L\big(\sigma+it_1,\chi^{(8d)} \big) \big|^{a_1} \cdots \big| L\big(\sigma+it_{k},\chi^{(8d)}  \big) \big|^{a_{k}}
 \Phi \Big( \frac d{X} \Big) \\
   \ll & X(\log X)^{(a_1^2+\cdots +a_{k}^2)/4}  \\
   & \times \prod_{1\leq j<l \leq k} \big|\zeta(1+i(t_j-t_l)+\frac 1{\log X}) \big|^{a_ia_j/2}\big|\zeta(1+i(t_j+t_l)+\frac 1{\log X}) \big|^{a_ia_j/2}\prod_{1\leq j\leq k} \big|\zeta(1+2it_j+\frac 1{\log X}) \big|^{a^2_i/4+a_i/2}.
\end{align*}

Thus we may focus on $d \in S(j)$ with $j \geq 1$ and to establish \eqref{Lprodboundssmoothed} and it suffices to show that
\begin{align}
\label{sumovermj}
\begin{split}
\sum_{j=1}^{\mathcal{J}} \ & \sumstar_{\substack{(d,2)=1 \\d \in S(j) }}  \big| L\big(\sigma+it_1,\chi^{(8d)} \big) \big|^{a_1} \cdots \big| L\big(\sigma+it_{k},\chi^{(8d)}  \big) \big|^{a_{k}}  \Phi \Big( \frac d{X} \Big) \\
   \ll & X(\log X)^{(a_1^2+\cdots +a_{k}^2)/4} \\
   & \times \prod_{1\leq j<l \leq k} \big|\zeta(1+i(t_j-t_l)+\frac 1{\log X}) \big|^{a_ia_j/2}\big|\zeta(1+i(t_j+t_l)+\frac 1{\log X}) \big|^{a_ia_j/2}\prod_{1\leq j\leq k} \big|\zeta(1+2it_j+\frac 1{\log X}) \big|^{a^2_i/4+a_i/2}.
\end{split}
\end{align}

We shall deal with the sum over $j$ in \eqref{sumovermj} individually, i.e. by working with a fixed $j$ with $1 \leq j \leq \mathcal{J}$.  From \eqref{basicest}, 
\begin{align}
\label{zetaNbounds}
\begin{split}
 \big| L  \big(\sigma+it_1, & \chi^{(8d)} \big) \big|^{a_1} \cdots \big| L\big(\sigma+it_{k},\chi^{(8d)}  \big) \big|^{a_{k}}\Phi(\frac {d}{X}) \\
\ll & A(d)^{-a}\exp \left(\frac {(Q+1)a}{\alpha_j} \right)\exp \left(\sum_{p\leq X^{\alpha_j/2}} \frac{a^2h_1(p,\sigma, X^{\alpha_j})}{4p} \right) \exp \Big (
 a\sum^j_{l=1}{\mathcal M}_{l,j}(d)\Big )\Phi \Big( \frac d{X} \Big).
\end{split}
\end{align}

The Taylor formula with integral remainder implies that for any $z \in \rear$,
\begin{align*}
\begin{split}
 \Big|e^z-\sum^{n-1}_{j=0}\frac {z^j}{j!}\Big| =& \Big|\frac 1{(n-1)!}\int\limits^z_0e^t(z-t)^{n-1} \dif t\Big|=\Big|\frac {z^n}{(n-1)!}\int\limits^1_0 e^{zs}(1-s)^{n-1} \dif s\Big| \\
\leq & \frac {|z|^n}{n!}\max (1, e^{\Re z}) \leq \frac {|z|^n}{n!}e^z \max (e^{-z}, e^{\Re z-z}) \leq \frac {|z|^n}{n!}e^z e^{|z|}.
\end{split}
 \end{align*}

This computation reveals that
\begin{align}
\label{ezrelation}
\begin{split}
  \sum^{n-1}_{j=0}\frac {z^j}{j!} =e^z \Big( 1+O \Big( \frac {|z|^n}{n!}e^{|z|} \Big) \Big).
\end{split}
 \end{align}

   We set $z=a {\mathcal M}_{l,j}(d)$, $n=\ell_l+1$ in \eqref{ezrelation} and apply \eqref{Stirling} to deduce that
\begin{align*}
\begin{split}
 E_{\ell_l}( a{\mathcal M}_{l,j}(d))  =\exp \Big ( a {\mathcal M}_{l,j}(d) \Big )\left( 1+ O\Big (\exp ( |a{\mathcal M}_{l,j}(d)| )
 \left( \frac{e|a{\mathcal M}_{l,j}(d)|}{ \ell_l+1} \right)^{ \ell_l+1}  \Big)  \right).
\end{split}
\end{align*}

  As $|a {\mathcal M}_{l,j}(d)| \le \ell_l/10^3$ for $d \in \mathcal{S}(j)$, the above simplifies to
\begin{align*}
 E_{\ell_l}( a{\mathcal M}_{l,j}(d)) =\exp \Big ( a {\mathcal M}_{l,j}(d) \Big ) \left( 1+   O(e^{-\ell_l}) \right).
\end{align*}

Hence
\begin{align*}
\begin{split}
\exp \Big ( a {\mathcal M}_{l,j}(d) \Big )  =E_{\ell_l}( a{\mathcal M}_{l,j}(d)) \left( 1+   O(e^{-\ell_l}) \right).
\end{split}
 \end{align*}

Inserting the above into \eqref{zetaNbounds}, we get
\begin{align*}
\begin{split}
  & \big| L\big(\sigma+it_1,\chi^{(8d)} \big) \big|^{a_1} \cdots \big| L\big(\sigma+it_{k},\chi^{(8d)}  \big) \big|^{a_{k}} \Phi(\frac {d}{X}) \\
\ll &  \exp \left(\frac {(Q+1)a}{\alpha_j} \right)\exp \left(\sum_{p\leq X^{\alpha_j/2}} \frac{a^2h_1(p,\sigma, X^{\alpha_j})}{4p} \right)A(d)^{-a}\prod^{j}_{l=1} \Big (1+O\big(e^{-\ell_l}\big)\Big )\prod^{j}_{l=1} E_{\ell_l}( a{\mathcal M}_{l,j}(d)) \Phi \Big( \frac d{X} \Big) \\
\ll & \exp \left(\frac {(Q+1)a}{\alpha_j} \right)\exp \left(\sum_{p\leq X^{\alpha_j/2}}\frac{a^2h_1(p,\sigma, X^{\alpha_j})}{4p} \right) A(d)^{-a}\prod^{j}_{l=1}E_{\ell_l}( a{\mathcal M}_{l,j}(d)) \Phi \Big( \frac d{X} \Big) .
\end{split}
\end{align*}
Note here, using $1+x \leq e^x$ for all $x \in \rear$ and the bounds in \eqref{sumoverell}, we get, for some large constant $C$,
\[ \prod^{j}_{l=1} \Big (1+O\big(e^{-\ell_l}\big) \Big ) \ll \exp \Big( C \sum_l e^{-\ell_l}\Big) \ll \exp \Big( C \sum_l \frac 1{\ell_l}\Big) \ll 1 , \]
where the last estimation above follows by noting that $\displaystyle \sum_l \frac 1{\ell_l}$ converges (being essentially a geometric series). \newline

   We then deduce from the description on $\mathcal{S}(j)$ and the above that when $j \geq 1$,
\begin{align}
\label{upperboundprodE0}
\begin{split}
\sumstar_{\substack{(d,2)=1 \\d \in S(j) }}  \big| L &\big(\sigma+it_1,\chi^{(8d)} \big) \big|^{a_1} \cdots \big| L\big(\sigma+it_{k},\chi^{(8d)}  \big) \big|^{a_{k}} \Phi \Big( \frac d{X} \Big)  \\
\ll & \exp \left(\frac {(Q+1)a}{\alpha_j} \right)\exp \left(\sum_{p\leq X^{\alpha_j/2}} \frac{a^2h_1(p,\sigma, X^{\alpha_j})}{4p} \right) \sum^{ \mathcal{J}}_{u=j+1}S_u,
\end{split}
\end{align}
  where
\begin{align*}
\begin{split}
S_u =: \sumstar_{\substack{(d,2)=1 }}A(d)^{-a}\prod^{j}_{l=1}E_{\ell_l}(a{\mathcal M}_{l,j}(d))  \Big ( \frac {10^3}{\ell_{j+1}}|a {\mathcal M}_{j+1,u}(d)|\Big)^{2\lceil 1/(10\alpha_{j+1})\rceil } \Phi \Big( \frac d{X} \Big).
\end{split}
\end{align*}

  We now focus on the evaluation of $S_u$ for a fixed $u$ by expanding the factors in $S_u$ into Dirichlet series.  To this end, we define functions $b_j(n)$, $1 \leq j \leq \mathcal{J}$ such that $b_j(n)=0$ or $1$ and $b_j(n)=1$ if and only if $\Omega(n) \leq \ell_j$ and the primes dividing $n$ are all from the interval $P_j$.
    We use these notations to write $E_{\ell_l}( a{\mathcal M}_{l,j}(d))$ as
\begin{equation*}
E_{\ell_l}( a{\mathcal M}_{l,j}(d)) = \sum_{n_l} \frac{h(n_l,\sigma, X^{\alpha_l})s(n_l,X^{\alpha_j})}{\sqrt{n_l}} \frac{a^{\Omega(n_l)}}{w(n_l)}  b_l(n_l)\chi^{(8d)}(n_l) , \quad 1\le l\le j.
\end{equation*}
If $X$ is large enough,
\begin{align} \label{sumoverell}
 \mathcal{J} \ll \log \log \log X, \quad \sum^{\mathcal{J}}_{j=1}\ell_j \leq 4e^B(\log \log X)^{3/2} \quad \mbox{and} \quad \sum^{\mathcal{J}}_{j=1} \alpha_{j}\ell_j \leq 40e^B 10^{-M/4}.
\end{align}
 It follows that $E_{\ell_l}( k{\mathcal M}_{l,j}(d)) $ is a short Dirichlet polynomial since $b_l(n_l)=0$ unless $n_l \leq
    (X^{\alpha_l})^{\ell_l}$. This implies that $\prod^{j}_{l=1}E_{\ell_l}( a{\mathcal M}_{l,j}(d))$ is also a short Dirichlet
    polynomial whose length is at most $X^{\sum^{\mathcal{J}}_{j=1} \alpha_{j}\ell_j} < X^{40e^B 10^{-M/4}}$ by \eqref{sumoverell}. We then write for simplicity that
\begin{align}
\label{Eprodexpression}
 \prod^{j}_{l=1}E_{\ell_l}( a{\mathcal M}_{l,j}(d))= \sum_{n  \leq X^{40e^B 10^{-M/4}}} \frac{x_n}{\sqrt{n}} \chi^{(8d)}(n),
\end{align}
where $x_n$ denotes the coefficient of the Dirichlet polynomial resulting from the expansion of the left-hand side of \eqref{Eprodexpression}.  Now \eqref{sumoverell} and the observation that $s(n,X^{\alpha_j})\leq 1$, together with \eqref{hbound}, lead to
\begin{align}
\label{xnbounds}
 x_n \ll a^{\sum^{j}_{l=1}\ell_l} \ll X^{\varepsilon}.
\end{align}

 On the other hand, similar to \eqref{M1fbound},
\begin{align}
\label{Mjuexp}
\begin{split}
 & \Big (   \frac {10^3}{\ell_{j+1}}|a{\mathcal M}_{j+1,u}(d)|\Big )^{2\lceil 1/(10^3\alpha_{j+1})\rceil } \\
= &  \Big ( \frac {a \cdot 10^3}{\ell_{j+1}} \Big)^{2\lceil /(10^3\alpha_{j+1})\rceil } \sum_{ \substack{ n_u \\ \Omega(n_u) = 2\lceil 1/(10^3\alpha_{j+1})\rceil \\ p|n_u \implies
p \in P_{j+1}}}
\frac{(2\lceil 1/(10^3\alpha_{j+1})\rceil  )!s(n_u, X^{\alpha_u})}{\sqrt{n_u}}\frac{h(n_u,\sigma, X^{\alpha_u})
  }{w(n_u)}\chi^{(8d)}(n_u) =:  \sum_{n_u} \frac {y_{n_u}}{\sqrt{n_u}}\chi^{(8d)}(n_u).
\end{split}
\end{align}
Note that $n_u \leq (X^{\alpha_{j+1}})^{2\lceil 1/(10^3\alpha_{j+1})\rceil } \leq X^{1/10^2}$.  The inequality $s(n,X^{\alpha_u})\leq 1$ and \eqref{hbound} again,
together with \eqref{Stirling}, reveals
\begin{align}
\label{ynbound}
\begin{split}
   y_{n_u} \ll \lceil 1/(10^3\alpha_{j+1})\rceil \Big ( \frac {2a\cdot 10^3 \lceil 1/(10^3\alpha_{j+1})\rceil }{e\ell_{j+1}} \Big)^{2\lceil 1 /(10^3\alpha_{j+1})\rceil } \ll X^{\varepsilon}.
\end{split}
\end{align}

  We multiply the Dirichlet series in \eqref{Eprodexpression} and \eqref{Mjuexp} together, keeping in mind that $n_l, n_u$ are mutually co-prime. This way, we may write $S_u$ as
\begin{align}
\label{Suexpression}
 S_u= \sum_{n  \leq X^{1/10}} v_n  \sumstar_{\substack{(d,2)=1  }} A(d)^{-a}\chi^{(8d)}(n) \Phi\Big( \frac{d}{X} \Big),
\end{align}
  where $|v_{n}| \ll  X^{\varepsilon}$ by \eqref{xnbounds} and \eqref{ynbound}. \newline

Lemma \ref{prsum} can be used to evaluate the inner-most sum in \eqref{Suexpression}.  The contribution from the $O$-term in \eqref{sumAd} is
\begin{align*}
 \ll X^{1/2+\varepsilon} \sum_{n  \leq X^{1/10}} X^{\varepsilon} n^{1/4+\varepsilon} \ll X,
\end{align*}
  which is negligible. \newline

 It then suffices to focus on what springs from the main term in \eqref{sumAd}.  This is
\begin{align*}
\begin{split}
  \ll X \prod^j_{l=1} & \Big (\sum_{\substack{n_l=\square}} \frac{h(n_l,\sigma, X^{\alpha_j})s(n_l,X^{\alpha_j})}{\sqrt{n_l}} \frac{a^{\Omega(n_l)}}{w(n_l)}  b_l(n_l) \prod_{p|n_l}(1+A(p)^{-a}p^{-1})^{-1}\Big )  \Big ( \Big ( \frac {a \cdot 10^3}{\ell_{j+1}} \Big)^{2\lceil 1/(10^3\alpha_{j+1})\rceil }  \\
& \times \sum_{ \substack{ n_u=\square \\ \Omega(n_u) = 2\lceil 1/(10^3\alpha_{j+1})\rceil \\ p|n_u \implies
p \in P_{j+1}}}
\frac{(2\lceil 1/(10^3\alpha_{j+1})\rceil  )!s(n_u, X^{\alpha_u})}{\sqrt{n_u}}\frac{h(n_u,\sigma, X^{\alpha_u})
  }{w(n_u)}\prod_{p|n_u}(1+A(p)^{-a}p^{-1})^{-1} \Big ).
\end{split}
\end{align*}

   Note that $\prod_{p|n_l}(1+A(p)^{-k}p^{-1})^{-1} \leq 1$ and $h(n, \sigma, X^{\alpha_j}), h(n, \sigma, X^{\alpha_u}) \geq 0$ when $n=\square$. It follows that the expression above is
\begin{align}
\label{Sumainterm}
\begin{split}
  \ll X \prod^j_{l=1} & \Big (\sum_{\substack{n_l=\square}} \frac{h(n_l,\sigma, X^{\alpha_j})s(n_l,X^{\alpha_j})}{\sqrt{n_l}} \frac{a^{\Omega(n_l)}}{w(n_l)}  b_l(n_l) \Big ) \Big ( \Big ( \frac {a \cdot 10^3}{\ell_{j+1}} \Big)^{2\lceil 1 /(10^3\alpha_{j+1})\rceil } \\
& \times \sum_{ \substack{ n_u=\square \\ \Omega(n_u) = 2\lceil 1/(10^3\alpha_{j+1})\rceil \\ p|n_u \implies
p \in P_{j+1}}}
\frac{(2\lceil 1/(10^3\alpha_{j+1})\rceil  )!s(n_u, X^{\alpha_u})}{\sqrt{n_u}}\frac{h(n_u,\sigma, X^{\alpha_u})
  }{w(n_u)}\Big ).
\end{split}
\end{align}

  We evaluate
\begin{align}
\label{sumsqurei}
   \sum_{\substack{n_l=\square}} \frac{h(n_l,\sigma, X^{\alpha_j})s(n_l,X^{\alpha_j})}{\sqrt{n_l}} \frac{a^{\Omega(n_l)}}{w(n_l)}  b_l(n_l)
\end{align}
  by noting that the factor $b_l(n_l)$ limits $n_l$ to have all prime factors in $P_l$ such that $\Omega(n_l) \leq \ell_l$. If we remove the restriction on $\Omega(n_l)$, then \eqref{sumsqurei} becomes
\begin{align*}
 \sum_{\substack{n_l=\square \\ p | n \Rightarrow p \in P_l}} \frac{h(n_l,\sigma, X^{\alpha_j})s(n_l,X^{\alpha_j})}{\sqrt{n_l}} \frac{a^{\Omega(n_l)}}{w(n_l)}.
\end{align*}
     On the other hand, using Rankin's trick by noticing that $2^{n_l-\ell_l}\ge 1$ if $\Omega(n_l) > \ell_l$, we see that the error incurred in this relaxation does not exceed
\begin{align}
\label{sumsqureierror1}
 \sum_{\substack{n_l=\square \\ p | n \Rightarrow p \in P_l}} \frac{2^{\Omega(n_l)-\ell_l} |h(n_l,\sigma, X^{\alpha_j})s(n_l,X^{\alpha_j})|}{\sqrt{n_l}} \frac{a^{\Omega(n_l)}}{w(n_l)}.
\end{align}

    Note that both the main term and the error term  above are now multiplicative functions of $n$ and hence can be evaluated in terms of products over primes.
    From Lemma \ref{RS}, the bound $|s(n_l, X^{\alpha_j})| \leq 1$ for all $n$ involved and \eqref{hbound}, emerges the following formula,
\begin{align*}
\begin{split}
 \sum_{\substack{n_l=\square \\ p | n \Rightarrow p \in P_l}} \frac{h(n_l,\sigma, X^{\alpha_j})s(n_l,X^{\alpha_j})}{\sqrt{n_l}} \frac{a^{\Omega(n_l)}}{w(n_l)}
 =& \prod_{p \in  P_l }\Big  (1+\frac {a^2h^2(p,\sigma, X^{\alpha_j})s^2(p,X^{\alpha_j})}{2p}+O\Big( \frac 1{p^{2}} \Big) \Big ) \\
 =& \prod_{p \in  P_l }\Big  (1+\frac {a^2h^2(p,\sigma, X^{\alpha_j})}{2p}+O\Big( \frac {\log p}{p\log X^{\alpha_j}}+\frac 1{p^{2}}  \Big) \Big ).
\end{split}
\end{align*}
Here the last expression above follows from \eqref{hbound} and the observation that for $p \leq X^{\alpha_j}$,
\begin{align*}
  s(p, X^{\alpha_j})=1+O\Big( \frac {\log p}{\log X^{\alpha_j}} \Big ).
\end{align*}

Apply the well-known inequality $1+x \leq e^x$ for any $x\in \rear$, we get
\begin{align*}
\begin{split}
  \prod_{p \in  P_l }\Big  (1+\frac {a^2h^2(p,\sigma, X^{\alpha_j})}{p}+O\Big( \frac {\log p}{p\log X^{\alpha_j}}+\frac 1{p^{2}}  \Big) \Big )
\leq \exp \Big ( \sum_{p \in  P_l } \frac {a^2h^2(p,\sigma, X^{\alpha_j})}{2p}+O\Big( \sum_{p \in  P_l } \Big (\frac {\log p}{p\log X^{\alpha_j}}+\frac 1{p^{2}}  \Big) \Big)\Big ).
\end{split}
\end{align*}

Similarly, via \eqref{sumpj}, \eqref{sumsqureierror1} is
\begin{align*}
 \leq 2^{-\ell_l/2}\exp \Big ( \sum_{p \in  P_l } \frac {a^2h^2(p,\sigma, X^{\alpha_j})}{2p}+O\Big( \sum_{p \in  P_l } \Big (\frac {\log p}{p\log X^{\alpha_j}}+\frac 1{p^{2}}  \Big) \Big)\Big ).
\end{align*}

   It follows that
\begin{align}
\label{sumsnlest}
 \sum_{\substack{n_l=\square}} \frac{h(n_l,\sigma, X^{\alpha_j})s(n_l,X^{\alpha_j})}{\sqrt{n_l}} \frac{a^{\Omega(n_l)}}{w(n_l)}  b_l(n_l) \leq
  \Big (1+O(2^{-\ell_l/2})\Big )\exp \Big ( \sum_{p \in  P_l } \frac {a^2h^2(p,\sigma, X^{\alpha_j})}{2p}+O\Big( \sum_{p \in  P_l } \Big (\frac {\log p}{p\log X^{\alpha_j}}+\frac 1{p^{2}}  \Big) \Big)\Big ).
\end{align}

  Next, using arguments that lead to \eqref{S0bound2} (the computation here is, in fact, simpler as the inner-most sum over $d$ on the right-hand side of \eqref{M1fbound} is replaced by $1$), upon taking $B$ large enough,
\begin{align}
\label{Su1stestmain}
\begin{split}
   \Big ( \frac {a \cdot 10^3}{\ell_{j+1}} \Big)^{2\lceil 1/(10^3\alpha_{j+1})\rceil } \sum_{ \substack{ n_u=\square \\ \Omega(n_u) = 2\lceil 1/(10^3\alpha_{j+1})\rceil \\ p|n_u \implies
p \in P_{j+1}}}
\frac{(2\lceil 1/(10^3\alpha_{j+1})\rceil  )!s(n_u, X^{\alpha_u})}{\sqrt{n_u}}\frac{h(n_u,\sigma, X^{\alpha_u})
  }{w(n_u)} \ll & e^{-10^3(Q+1)a/(2\alpha_{j+1})}.
\end{split}
\end{align}

 We conclude from \eqref{Suexpression}, \eqref{Sumainterm}, \eqref{sumsnlest} and \eqref{Su1stestmain} that
\begin{align}
\label{prodEestmain}
\begin{split}
  S_u
\leq &  e^{-10^3(Q+1)a/(2\alpha_{j+1})}X\prod^j_{l=1}\Big (1+O(2^{-\ell_l/2})\Big ) \times \exp \Big (\sum_{p \in  \bigcup^j_{l=1}P_l }\frac {a^2h^2(p,\sigma, X^{\alpha_j})}{2p}
 +O\Big(\sum_{p \in  \bigcup^j_{l=1}P_l }\Big (\frac {\log p}{p\log X^{\alpha_j}}+\frac 1{p^{2}}\Big ) \Big )\Big ) \\
\ll & e^{-10^3(Q+1)a/(2\alpha_{j+1})}X \exp \Big (\sum_{p \in  \bigcup^j_{l=1}P_l }\frac {a^2h^2(p,\sigma, X^{\alpha_j})}{2p}
 +O\Big(\sum_{p \in  \bigcup^j_{l=1}P_l }\Big (\frac {\log p}{p\log X^{\alpha_j}}+\frac 1{p^{2}}\Big ) \Big )\Big ).
\end{split}
\end{align}

    Note that we have
\begin{align}
\label{sumpsqure}
  \sum_{p \in  \bigcup^{j}_{l=1}P_l }\frac 1{p^{2}} \ll 1.
\end{align}
  Moreover, by Lemma \ref{RS},
\begin{align}
\label{sumsplogp}
  \sum_{p \in  \bigcup^j_{l=1}P_l }\frac {\log p}{p\log X^{\alpha_j}}=& \sum_{p \leq  X^{\alpha_j} }\frac {\log p}{p\log X^{\alpha_j}} \ll 1.
\end{align}

Thus from \eqref{prodEestmain}--\eqref{sumsplogp},
\begin{align*}
  S_u \ll & e^{-10^3(Q+1)a/(2\alpha_{j+1})}X \exp \Big (\sum_{p \in  \bigcup^j_{l=1}P_l }\frac {a^2h^2(p,\sigma, X^{\alpha_j})}{2p} \Big ).
\end{align*}

Inserting the above into \eqref{upperboundprodE0} and utilizing the observation that $20/\alpha_{j+1}=1/\alpha_j$, we uncover that
\begin{align}
\label{produpperboundoverSj}
\begin{split}
& \sumstar_{\substack{(d,2)=1 \\d \in S(j) }} \big| L\big(\sigma+it_1,\chi^{(8d)} \big) \big|^{a_1} \cdots \big| L\big(\sigma+it_{k},\chi^{(8d)}  \big) \big|^{a_{k}}  \Phi \Big( \frac {d}{X} \Big) \\
\ll & (\mathcal{J}-j)\exp \left(-\frac {10(Q+1)a}{\alpha_j} \right)X\exp \Big (\sum_{p \in  \bigcup^j_{l=1}P_l }\frac {a^2h^2(p,\sigma, X^{\alpha_j})}{2p}+\sum_{p\leq X^{\alpha_j/2}}
\frac{a^2h_1(p,\sigma, X^{\alpha_j})}{4p} \Big )\\
\ll & \exp \left(-\frac {(Q+1)a}{\alpha_j} \right)X
\exp \Big (\sum_{p \in  \bigcup^j_{l=1}P_l }\frac {a^2h^2(p,\sigma, X^{\alpha_j})}{2p}+\sum_{p\leq X^{\alpha_j/2}}
\frac{a^2h_1(p,\sigma, X^{\alpha_j})}{4p} \Big ),
\end{split}
\end{align}
 where the last estimation above follows from \eqref{sumpj}. \newline

   We now specify $\sigma=1/2$ in \eqref{produpperboundoverSj} and the observation that $|h(p)| \ll 1$ for any prime $p$ to see that
\begin{align}
\label{sumpest}
\begin{split}
 \sum_{p \in  \bigcup^j_{l=1}P_l }\frac {a^2h^2(p,\sigma, X^{\alpha_j})}{2p}=& \sum_{p \leq X^{\alpha_j}}\frac {(2h(p))^2)}{2p^{1+1/\log X^{\alpha_j}}} = \sum_{p \leq X^{\alpha_j}}\frac {(2h(p))^2)}{2p}
+O\Big(\sum_{p \leq X^{\alpha_j}}\Big|(2h(p))^2\Big |\Big |\frac {1}{p}-\frac {1}{p^{1+1/\log X^{\alpha_j}}}\Big|\Big) \\
=& \sum_{p \leq X^{\alpha_j}}\frac {(2h(p))^2)}{2p}
+O\Big(\sum_{p \leq X^{\alpha_j}}\frac {\log p}{p\log X^{\alpha_j}}\Big) = \sum_{p \leq X^{\alpha_j}}\frac {(2h(p))^2)}{2p} +O(1),
\end{split}
\end{align}
  where the last estimation above follows from \eqref{sumsplogp}. \newline

  Similarly, we have
\begin{align}
\label{sump2est}
\begin{split}
 \sum_{p\leq X^{\alpha_j/2}}
\frac{a^2h_1(p,\sigma, X^{\alpha_j})}{4p}=& \sum_{p\leq X^{\alpha_j/2}} \frac{h(p^2)}{p}
+O(1). 
\end{split}
\end{align}

Using \eqref{sumpest} and \eqref{sump2est} in \eqref{produpperboundoverSj} lead to
\begin{align}
\label{produpperboundoverSjsimplified}
\sumstar_{\substack{(d,2)=1 \\d \in S(j) }} \big| L\big(\sigma+it_1,\chi^{(8d)} \big) \big|^{a_1} \cdots \big| L\big(\sigma+it_{k},\chi^{(8d)}  \big) \big|^{a_{k}} \Phi \Big( \frac {d}{X} \Big) \ll X \exp \Big (-\frac {(Q+1)a}{\alpha_j} + \sum_{p \leq X}\frac {(2h(p))^2)}{2p}+\sum_{p\leq X} \frac{h(p^2)}{p} \Big ).
\end{align}

   Note that
\begin{align}
\label{hexp}
\begin{split}
& \sum_{p \leq X }\frac {(2h(p))^2}{2p}+\sum_{p\leq X} \frac{h(p^2)}{p} = \sum_{p \leq X }\frac {1}{2p}\Big ( \Big( \sum^{k}_{j=1}a_j\cos(t_j\log p) \Big)^2+\sum^{k}_{j=1}a_j\cos(2t_j\log p)\Big ) \\
=& \sum_{p \leq X }\frac {1}{2p}\Big (\sum^{k}_{j=1}a^2_j\cos^2(t_j\log p)+2\sum_{1 \leq i<j \leq k}a_ia_j\cos(t_i\log p)\cos(t_j\log p)+\sum^{k}_{j=1}a_j\cos(2t_j\log p)\Big ) \\
=& \sum_{p \leq X }\frac {1}{2p}\Big (\frac 12\sum^{k}_{j=1}a^2_j+\sum_{1 \leq i<j \leq k}a_ia_j\big(\cos((t_i+t_j)\log p)+\cos((t_i-t_j)\log p)\big )+\sum^{k}_{j=1} \Big( \frac {a^2_j}{2}+a_j \Big) \cos(2t_j\log p)\Big ),
\end{split}
\end{align}
upon using the relation
\begin{align*}
\begin{split}
& \cos (\alpha)\cos (\beta)=\frac 12 (\cos (\alpha+\beta)+\cos(\alpha-\beta)).
\end{split}
\end{align*}

  We now apply \eqref{mertenstype} to evaluate the last expression in \eqref{hexp} and then substitute the result into \eqref{produpperboundoverSjsimplified}.  Summing over $j$, which then leads to a convergent series. This allows us to derive the desired estimation in \eqref{sumovermj} and hence completes the proof of Theorem \ref{t1}.

\section{Proof of Theorem \ref{quadraticmean}}
\label{sec: mainthm}

\subsection{Initial treatments}

As explained in the paragraph below Theorem \ref{quadraticmean}, it suffices to establish \eqref{mainestimation}. To that end, we note first that by \cite[Theorem 1]{Armon} that we have for any $m>0$,
\begin{align*}
\begin{split}
S_m(X,Y) \ll & X^{m+1}.
\end{split}
\end{align*}

Thus, we may assume that $Y \leq X$.  Let $\Phi_U(t)$ be a non-negative smooth function supported on $(0,1)$, satisfying $\Phi_U(t)=1$ for $t \in (1/U, 1-1/U)$ with $U$ a parameter to be optimized later and $\Phi^{(j)}_U(t) \ll_j U^j$ for all integers $j \geq 0$. We denote the Mellin transform of $\Phi_U$ by $\widehat{\Phi}_U$ and observe that repeated integration by parts (keeping in mind that the function $\Phi^{(j)}_U(t)$ is supported on intervals with total lengths being bounded by $O(1/U)$ for each $j \geq 1$) gives that, for any integer $E \geq 1$ and $\Re(s) \geq 1/2$,
\begin{align}
\label{whatbound}
 \widehat{\Phi}_U(s)  \ll  U^{E-1}(1+|s|)^{-E}.
\end{align}

    Note that H\"older’s inequality implies $|x + y|^{2m} \leq 2^{2m-1}(|x|^{2m} + |y|^{2m})$ for any $x, y \in \mc$. We insert the function $\Phi_U(\frac nY)$ into the definition of $S_m(X,Y)$ and obtain that
\begin{align}
\label{charintinitial}
\begin{split}
S_m(X,Y) \ll &
  \sumstar_{\substack{d \leq X \\ (d,2)=1}}\Big | \sum_{n}\chi^{(8d)}(n)\Phi_U \Big( \frac {n}{Y} \Big)\Big |^{2m}
  +\sumstar_{\substack{d \leq X \\ (d,2)=1}}\bigg|\sum_{n\leq Y} \chi^{(8d)}(n)\Big(1-\Phi_U \Big( \frac {n}{Y} \Big) \Big)\bigg|^{2m}.
\end{split}
\end{align} 

The Mellin inversion gives that
\begin{align}
\label{charintinitial1}
\begin{split}
 \sumstar_{\substack{d \leq X \\ (d,2)=1}}\Big | \sum_{n}\chi^{(8d)}(n)\Phi_U \Big( \frac {n}{Y} \Big)\Big |^{2m}=&
\sumstar_{\substack{d \leq X \\ (d,2)=1}}\Big | \int\limits_{(2)}L(s, \chi^{(8d)})Y^s\widehat{\Phi}_U(s) \dif s\Big |^{2m}.
\end{split}
\end{align} 

Observe that by \cite[Corollary 5.20]{iwakow}, that under GRH, for every primitive Dirichlet character $\chi$ modulo $q$, $\Re(s) \geq 1/2$ and any $\varepsilon>0$,
\begin{align}
\label{Lbound}
 L(s, \chi) \ll |qs|^{\varepsilon}.
\end{align}

   The bounds in  \eqref{whatbound} and \eqref{Lbound} allow us to shift the line of integration in \eqref{charintinitial1} to $\Re(s)=1/2$ to obtain that
\begin{align}
\label{charint}
\begin{split}
  \sumstar_{\substack{d \leq X \\ (d,2)=1}}\Big | \sum_{n}\chi^{(8d)}(n)\Phi_U \Big( \frac {n}{Y} \Big)\Big |^{2m} = &
 \sumstar_{\substack{d \leq X \\ (d,2)=1}}\Big | \int\limits_{(1/2)}L(s, \chi^{(8d)})Y^s\widehat{\Phi}_U(s) \dif s\Big |^{2m}.
\end{split}
\end{align} 	

  We split the range of integral in \eqref{charint} into two parts, $|\Im(s)| \leq X^Q$ and $|\Im(s)| > X^Q$, for some $Q>0$ to be specified later, obtaining that \eqref{charint} is
\begin{align}
\label{charintinitial0}
\ll \sumstar_{\substack{d \leq X \\ (d,2)=1}}\Big | \int\limits_{\substack{ (1/2) \\ |\Im(s)| \leq X^Q}}L(s, \chi^{(8d)})Y^s\widehat{\Phi}_U(s) \dif s\Big |^{2m}+
\sumstar_{\substack{d \leq X \\ (d,2)=1}}\Big | \int\limits_{\substack{ (1/2) \\ |\Im(s)| > X^Q}}L(s, \chi^{(8d)})Y^s\widehat{\Phi}_U(s) \dif s\Big |^{2m}.
\end{align} 	

  We next apply \eqref{whatbound} and H\"older's inequality to deduce that, as $m \geq 1/2$, 
\begin{align}
\label{LintImslarge}
  \sumstar_{\substack{d \leq X \\ (d,2)=1}} & \Big | \int\limits_{\substack{ (1/2) \\ |\Im(s)| > X^Q}}L(s, \chi^{(8d)})Y^s\widehat{\Phi}_U(s) \dif s\Big |^{2m}
 \ll Y^m\Big (\int\limits_{\substack{ (1/2) \\ |\Im(s)| > X^Q}}\Big | \widehat{\Phi}_U(s) \Big| |\dif s| \Big )^{2m-1}
 \int\limits_{\substack{ (1/2) \\ |\Im(s)| > X^Q}}\sumstar_{\substack{d \leq X \\ (d,2)=1}}\Big |L(s, \chi^{(8d)})\Big|^{2m} \Big| \widehat{\Phi}_U(s)\Big | |\dif s|.
\end{align} 	

By \eqref{whatbound},
\begin{align}
\label{phiint}
\begin{split}
  \int\limits_{\substack {(1/2) \\ |\Im(s)| > X^Q}}\Big | \widehat{\Phi}_U(s) \Big| |\dif s|
 \ll & \int\limits_{\substack {(1/2) \\ |\Im(s)| > X^Q}}\frac U{1+|s|^2} |\dif s| \ll_Q \frac {U}{X^Q}.
\end{split}
\end{align}

Using \eqref{Lbound},
\begin{align}
\label{Lintslarge}
\begin{split}
  & \int\limits_{\substack {(1/2) \\ |\Im(s)| > X^Q}}\sumstar_{\substack{d \leq X \\ (d,2)=1}}
  \Big | L(s, \chi^{(8d)})\Big |^{2m} \cdot \Big |\widehat{\Phi}_U(s)\Big | |\dif s| \ll  \int\limits_{\substack {(1/2) \\ |\Im(s)| > X^Q}}\sumstar_{\substack{d \leq X \\ (d,2)=1}}|ds|^{\varepsilon}\cdot\frac U{1+|s|^2}  |\dif s| \ll XUX^{-Q(1-\varepsilon)+\varepsilon}.
\end{split}
\end{align}

   We now set $U=X^{2\varepsilon}$ and $Q=4\varepsilon$.  From \eqref{charintinitial}, \eqref{charint}--\eqref{Lintslarge},
\begin{align}
\label{Ssimplified}
\begin{split}
S_m(X,Y) \ll &
   \sumstar_{\substack{d \leq X \\ (d,2)=1}}\Big | \int\limits_{\substack{ (1/2) \\ |\Im(s)| \leq X^{\varepsilon}}}L(s, \chi^{(8d)})Y^s\widehat{\Phi}_U(s) \dif s\Big |^{2m}+\sumstar_{\substack{d \leq X \\ (d,2)=1}}\bigg|\sum_{n\leq Y} \chi^{(8d)}(n)\Big (1-\Phi_U\Big( \frac {n}{Y} \Big)\Big )\bigg|^{2m}
   +O(XY^m) \\
   \ll & Y^m \sumstar_{\substack{d \leq X \\ (d,2)=1}}
   \Big | \int\limits_{\substack{ (1/2) \\ |t| \leq X^{\varepsilon}}}\Big |L( \tfrac{1}{2}+it, \chi^{(8d)})\Big |\frac 1{1+|t|} \dif t\Big |^{2m}+\sumstar_{\substack{d \leq X \\ (d,2)=1}}\bigg|\sum_{n\leq Y} \chi^{(8d)}(n)\Big (1-\Phi_U\Big( \frac {n}{Y} \Big)\Big )\bigg|^{2m}
   +O(XY^m).
\end{split}
\end{align}

Theorem \ref{quadraticmean} follows from the following Lemma.
\begin{lemma}
\label{Lsmooth}
With the notation as above and assume the truth of GRH. We have for any integer $k \geq 1$ and any real numbers $2m \geq k+1, \varepsilon>0$, 
\begin{equation}
\label{Lsmoothest}
\sumstar_{\substack{d \leq X \\ (d,2)=1}}
   \Big | \int\limits_{\substack{ (1/2) \\ |t| \leq X^{\varepsilon}}}\Big |L(1/2+it, \chi^{(8d)})\Big |\frac 1{1+|t|}dt\Big |^{2m} \ll X(\log X)^{ E(m,k,\varepsilon)},
\end{equation}
where $E(m,k,\varepsilon)$ is defined in \eqref{Edef}.
\end{lemma}

\begin{lemma}
\label{fdiff}
With the notation as above and assume the truth of GRH. We have for $m \geq 1$, 
\begin{equation}
\label{theorem3firstrest}
\sumstar_{\substack{d \leq X \\ (d,2)=1}}\bigg|\sum_{n\leq Y} \chi^{(8d)}(n)\Big (1-\Phi_U \Big( \frac {n}{Y} \Big)\Big )\bigg|^{2m} \ll XY^m.
\end{equation}
\end{lemma}

The remainder of the paper will be devoted to the proofs of these Lemmas.

\section{Proof of Lemma \ref{Lsmooth}}

  Our proof follows closely the treatments in \cite{Szab}. We deduce from \eqref{Ssimplified} by symmetry and H\"older's
inequality that for $a=1-1/(2m)+\varepsilon$ with $\varepsilon>0$, 
\begin{align*}
\begin{split}
 \Big | \int\limits_{ |t| \leq X^{\varepsilon}}  & \Big |L(\tfrac{1}{2}+it, \chi^{(8d)})|\frac {\dif t}{1+|t|} \Big |^{2m} \ll \Big |\int_0^{X^{\varepsilon}} \frac{|L(\tfrac{1}{2}+it,\chi^{(8d)})|}{t+1} \dif t\Big |^{2m} \\
 & \leq \bigg(\sum_{n\leq \log X+1} n^{-2am/(2m-1)} \bigg)^{2m-1}
    \sum_{n\leq  \log X+1} \bigg(n^a\int\limits_{e^{n-1}-1}^{e^{n}-1 } \frac{|L(\tfrac{1}{2}+it,\chi^{(8d)}) |}{t+1} \dif t\bigg)^{2m}   \\
  & \ll \sum_{n\leq  \log X+1} \frac{n^{2m-1+\varepsilon} }{e^{2nm} } \bigg( \int\limits_{e^{n-1}-1}^{e^{n}-1 } |L(\tfrac{1}{2}+it,\chi^{(8d)}) | \dif t \bigg)^{2m}.
\end{split}
\end{align*}

   We need the following result, which will be proved later, to estimate the last expression above.
\begin{proposition}
\label{t3prop}
 With the notation as above and the truth of GRH, we have for any fixed integer $k \geq 1$ and any real numbers $2m \geq k+1$, $10 \leq E=X^{O(1)}$,
\begin{align}
\label{finiteintest}
\begin{split}
  \sumstar_{\substack{d \leq X \\ (d,2)=1}} & \bigg(\int\limits_{0}^{E}|L(\tfrac{1}{2}+it,\chi^{(8d)})| \dif t\bigg)^{2m} \\
 \ll & X\big( (\log X)^{2m^2-m+1}E^k(\log \log E)^{O(1)}+(\log X)^{(2m-k)^2/4+1}E^{2m}(\log \log E)^{O(1)}(\log\log X)^{O(1)} \\
 & \hspace*{2cm} +(\log X)^{2m^2-2mk+3k^2/4+m-3k/4}E^{2m-1}(\log \log E)^{O(1)}(\log\log X)^{O(1)}\big).
\end{split}
\end{align}
\end{proposition}

We apply Proposition \ref{t3prop} to see that for any integer $k \geq 1$ and any real numbers $2m \geq k+1, \varepsilon>0$, 
\begin{align*}
  \sum_{n\leq  \log X+1} & \frac{n^{2m-1+\varepsilon} }{e^{2nm} }\sumstar_{\substack{d \leq X \\ (d,2)=1}} \bigg( \int_{e^{n-1}-1}^{e^{n}-1 } |L(1/2+it,\chi^{(8d)}) | dt \bigg)^{2m}
    \\
     \ll & X\sum_{n\leq  \log X+1} \frac{n^{2m-1+\varepsilon} }{e^{2nm} } \\
    & \times 
     \Big((\log X)^{2m^2-m+1}e^{kn}(\log 2n)^{O(1)}+(\log X)^{(2m-k)^2/4+1}(\log 2n)^{O(1)}(\log\log X)^{O(1)} e^{2mn} \\
    & \hspace*{2cm} +(\log X)^{2m^2-2mk+3k^2/4+m-3k/4}(\log 2n)^{O(1)}(\log\log X)^{O(1)} e^{(2m-1)n}\Big)  \ll X(\log X)^{E(m,k,\varepsilon)}.
\end{align*}

We now deduce from the above that \eqref{Lsmoothest} holds, completing the proof of the Lemma~\ref{Lsmooth}.

\begin{proof}[Proof of Proposition~\ref{t3prop}]
 We have by symmetry that for each $d$,
\begin{align}
\label{Lintdecomp}
    \bigg(\int\limits_{0}^{E} |L(\tfrac{1}{2}+ it, \chi^{(8d)})| \dif t\bigg)^{2m}
      \ll \int\limits_{[0,E]^k}\prod_{a=1}^k|L(1/2+ it_a, \chi^{(8d)})| \bigg(\int_{\mathcal{D} }|L(1/2+iu, \chi^{(8d)})| \dif u \bigg)^{2m-k} \dif \mathbf{t},
\end{align}
where $\mathcal{D}=\mathcal{D}(t_1,\ldots,t_k)=\{ u\in [0,E]:|t_1-u|\leq |t_2-u|\leq \ldots \leq |t_k-u| \}$. \newline

  We set $\mathcal{B}_1=\big[-\frac{1}{\log X},\frac{1}{\log X}\big]$, $\mathcal{B}_j=\big[-\frac{e^{j-1}}{\log X}, -\frac{e^{j-2}}{\log X}\big]
  \cup \big[\frac{e^{j-2}}{\log X}, \frac{e^{j-1}}{\log X}\big]$ for $2\leq j< \lfloor \log \log X\rfloor+10 =: K$ and $\mathcal{B}_K=[-E,E]\setminus \bigcup_{1\leq j<K} \mathcal{B}_j$. \newline

Observe that for any $t_1\in [0,E]$,  we have $\mathcal{D}\subset [0,E] \subset t_1+[-E,E]\subset \bigcup_{1\leq j\leq K} t_1+\mathcal{B}_j$. Thus
if we denote $\mathcal{A}_j=\mathcal{B}_j\cap (-t_1+\mathcal{D})$, then $(t_1+\mathcal{A}_j)_{1\leq j\leq K}$ form a partition of $\mathcal{D}$.  Using Hölder's inequality twice to deduce that
\begin{align}
\label{LintoverD}
\begin{split}
    & \bigg(\int\limits_{\mathcal{D}}|L(\tfrac{1}{2} + iu, \chi^{(8d)})| \dif u\bigg)^{2m-k}  \leq \bigg( \sum_{1\leq j\leq K} \frac{1}{j}\cdot  j \int\limits_{t_1+\mathcal{A}_j} |L( \tfrac{1}{2}+iu, \chi^{(8d)})| \dif u  \bigg)^{2m-k} \\
     \leq & \bigg(\sum_{1\leq j\leq K} j^{2m-k} \bigg( \int\limits_{t_1+\mathcal{A}_j} \big|L( \tfrac{1}{2}+iu, \chi^{(8d)})\big| \dif u  \bigg)^{2m-k}\bigg)
     \bigg(\sum_{1\leq j\leq K } j^{-(2m-k)/(2m-k-1)} \bigg)^{2m-k-1} \\
     \ll & \sum_{1\leq j\leq K} j^{2m-k} \bigg( \int\limits_{t_1+\mathcal{A}_j} |L( \tfrac{1}{2}+iu, \chi^{(8d)})| \dif u \bigg)^{2m-k} \leq \sum_{1\leq j\leq K} j^{2m-k} |\mathcal{B}_j|^{2m-k-1} \int\limits_{t_1+\mathcal{A}_j} |L(1/2+iu, \chi^{(8d)})|^{2m-k} \dif u.
\end{split}
\end{align}
For $\mathbf{t}=(t_1,\ldots,t_k)$, we write
$$L(\mathbf{t},u)=\sumstar_{\substack{d \leq X \\ (d,2)=1}}\prod_{a=1}^k|L( \tfrac{1}{2}+it_a, \chi^{(8d)})| \cdot |L( \tfrac{1}{2}+iu, \chi^{(8d)})|^{2m-k}.$$
  From \eqref{Lintdecomp} and \eqref{LintoverD}, we deduce
\begin{align}
\label{Lintest}
\begin{split}
    \sumstar_{\substack{d \leq X \\ (d,2)=1}}\bigg(\int\limits_{0}^{E}|L(1/2+it,\chi^{(8d)})|\dif t \bigg)^{2m}\ll &
    \sum_{1\leq l_0\leq K} l_0^{2m-k} |\mathcal{B}_{l_0}|^{2m-k-1} \int\limits_{[0,E]^k}\int\limits_{t_1+\mathcal{A}_{l_0}} L(\mathbf{t},u) \dif u \dif \mathbf{t}  \\
     \ll & \sum_{1\leq l_0, l_1, \ldots l_{k-1}\leq K} l_0^{2m-k} |\mathcal{B}_{l_0}|^{2m-k-1} \int\limits_{\mathcal{C}_{l_0,l_1, \cdots, l_{k-1}}} L(\mathbf{t},u) \dif u \dif \mathbf{t},
\end{split}
\end{align}
where
$$\mathcal{C}_{l_0,l_1, \cdots, l_{k-1}}=\{(t_1,\ldots,t_k,u)\in [0,E]^{k+1}: u\in t_1+ \mathcal{A}_{l_0},\, |t_{i+1}-u|-|t_i-u|\in \mathcal{B}_{l_i}, \ 1 \leq i \leq k-1\}.$$
We now separate two cases in the last summation of \eqref{Lintest} according to the size of $l_0$. \newline

\textbf{Case 1:} $l_0<K$. First, the volume of the region $\mathcal{C}_{l_0,l_1, \cdots, l_{k-1}}$ is $\ll  E^k e^{l_0+l_1+\cdots+l_{k-1}} (\log X)^{-k}$. Also,
by the definition of $\mathcal{C}_{l_0,l_1, \cdots, l_{k-1}}$ and the observation that $t_i, u \geq 0, 1\leq i \leq k$, $e^{l_0}/\log X \ll |t_1-u|\ll |t_1+u|\ll E =X^{O(1)}$ so that $g(|t_1\pm u|)\ll \log X \cdot \log \log E/e^{l_0}$, where we recall the definition of $g$ in \eqref{gDef}.  We deduce from the definition of $\mathcal{A}_j$ that $|t_2-u|\geq |t_1-u|$, so that $E \gg |t_2+u| \gg |t_2-u|= |t_1-u|+(|t_2-u|-|t_1-u|)\gg  e^{l_0}/\log X + e^{l_1}/\log X$, which implies that $g(|t_2\pm u|)\ll \log X \cdot \log \log E/e^{\max(l_0,l_1) }$.
Similarly, we have  $g(|t_i\pm u|)\ll \log X \cdot \log \log E /e^{\max(l_0,l_1,\ldots, l_{i-1}) }$ for any $1 \leq i \leq k$. 
Moreover, we have $\sum^{j-1}_{s=i}(|t_{s+1}-u|-|t_s-u|) \leq |t_j-t_i| \leq |t_j+t_i|$ for any $1 \leq i < j \leq k$, so that we have $g(|t_{j}\pm t_i|)\ll \log X \cdot \log \log E/e^{\max(l_i,\ldots, l_{j-1} ) }$.
We also bound $g(|2t_i|), 1\leq i \leq k$ and $g(|2u|)$ trivially by $\log X$ to deduce from Corollary \ref{cor1} that for $(t_1,\ldots,t_k,u)\in \mathcal{C}_{l_0,l_1, \cdots, l_{k-1}}$,
\begin{align*}
     & L(\mathbf{t},u) \\
     \ll & X(\log X)^{((2m-k)^2+k)/4+(2m-k)^2/4+(2m-k)/2+3k/4}(\log \log E)^{O(1)}
     \bigg(\prod^{k-1}_{i=0}\frac{\log X}{e^{ \max(l_0,l_1,\ldots, l_{i}) }} \bigg)^{2m-k}
     \bigg(\prod^{k-1}_{i=1} \prod^{k}_{j=i+1}\frac{\log X}{e^{\max(l_i,\ldots, l_{j-1} ) } } \bigg) \\
     = & X(\log X)^{m(2m+1)}(\log \log E)^{O(1)}  \exp\Big( -(2m-k)\sum^{k-1}_{i=0}\max(l_0,l_1,\ldots, l_{i})-\sum^{k-1}_{i=1} \sum^{k}_{j=i+1}\max(l_i,\ldots, l_{j-1} )\Big).
\end{align*}

 Moreover, $ |\mathcal{B}_{l_0}|\ll e^{l_0}/\log X$, so that
\begin{align}
\label{firstcase}
\begin{split}
       &  \sum_{\substack{1\leq l_0<K \\ 1\leq l_1, \ldots l_{k-1}\leq K}}  l_0^{2m-k} |\mathcal{B}_{l_0}|^{2m-k-1} \int\limits_{\mathcal{C}_{l_0,l_1, \cdots, l_{k-1}}} L(\mathbf{t},u) \dif u \dif \mathbf{t} \\
    \ll & X(\log X)^{2m^2-m+1}E^k(\log \log E)^{O(1)}  \cdot \\
    &  \times \sum_{\substack{1\leq l_0<K \\ 1\leq l_1, \ldots l_{k-1}\leq K}}  l_0^{2m-k}\exp\Big( (2m-k-1)l_0+\sum^{k-1}_{i=0}l_i-(2m-k)\sum^{k-1}_{i=0}\max(l_0,l_1,\ldots, l_{i})-\sum^{k-1}_{i=1} \sum^{k}_{j=i+1}\max(l_i,\ldots, l_{j-1} )\Big) \\
    = & X(\log X)^{2m^2-m+1}E^k(\log \log E)^{O(1)}  \cdot \\
    &  \times \sum_{\substack{1\leq l_0<K \\ 1\leq l_1, \ldots l_{k-1}\leq K}}  l_0^{2m-k}\exp\Big( -(2m-k)\sum^{k-1}_{i=1}\max(l_0,l_1,\ldots, l_{i})-\sum^{k-1}_{i=1} \sum^{k}_{j=i+2}\max(l_i,\ldots, l_{j-1} )\Big) \\
    \ll &   X(\log X)^{2m^2-m+1}E^k(\log \log E)^{O(1)} \sum_{\substack{1\leq l_0<K \\ 1\leq l_1, \ldots l_{k-1}\leq K}}  l_0^{2m-k}\exp\Big( -(2m-k)\sum^{k-1}_{i=1}\frac {\sum^{i}_{s=0}l_s}{i+1}-\sum^{k-1}_{i=1} \sum^{k}_{j=i+2}\frac {\sum^{j-1}_{s=i}l_s}{j-i}\Big)  \\
    \ll &   X(\log X)^{2m^2-m+1}E^k(\log \log E)^{O(1)}.
\end{split}
\end{align}

\textbf{Case 2} $l_0=K$.  For each $1 \leq i \leq k$, 
we have $g(|t_i\pm u|)\ll \log\log  E$. Also, similar to Case 1, we have $g(|t_j\pm t_i|) \ll \log X/e^{\max(l_i,\ldots, l_{j-1} ) }$ for $1 \leq i \leq k$. As $|t_1-u| \leq |t_1|+|u|$, we see that either $|u| \geq 5$ or $|u| \leq 5$. If $|u| \geq 5$, then we have $g(|2u|) \ll \log\log  E$. We further use the trivial bound $g(|2t_i|) \ll \log X$ with $1 \leq i \leq k$, obtaining
\begin{align*}
     L(\mathbf{t},u) \ll & X(\log X)^{((2m-k)^2+k)/4+3k/4}(\log \log E)^{O(1)}
     \bigg(\prod^{k-1}_{i=1} \prod^{k}_{j=i+1}\frac{\log X}{e^{\max(l_i,\ldots, l_{j-1} ) } } \bigg) \\
     = & X(\log X)^{(2m-k)^2/4+k}(\log \log E)^{O(1)}  \exp\Big( -\sum^{k-1}_{i=1} \sum^{k}_{j=i+1}\max(l_i,\ldots, l_{j-1} )\Big).
\end{align*}

    If $|u| \leq 5$, then using $|t_1-u| \leq |t_i-u| \leq |t_i|+|u|$ for any $1 \leq i \leq k$ we see that $|t_i| \geq 5$ so that $g(|2t_i|) \ll \log\log  E$.  Again the trivial bound of $g(|2u|) \ll \log X$ yields
\begin{align*}
 L(\mathbf{t},u) \ll & X(\log X)^{((2m-k)^2+k)/4+(2m-k)^2/4+(2m-k)/2}(\log \log E)^{O(1)}
     \bigg(\prod^{k-1}_{i=1} \prod^{k}_{j=i+1}\frac{\log X}{e^{\max(l_i,\ldots, l_{j-1} ) } } \bigg) \\
     = & X(\log X)^{2m^2-2mk+3k^2/4+m-3k/4}(\log \log E)^{O(1)}  \exp\Big( -\sum^{k-1}_{i=1} \sum^{k}_{j=i+1}\max(l_i,\ldots, l_{j-1} )\Big).
\end{align*}
The volume of the region $\mathcal{C}_{K,l_1, \cdots, l_{k-1}}$ is $\ll E^{k+1} e^{l_1+\cdots+l_{k-1}} (\log X)^{-k+1}$.  As $|\mathcal{B}_K|\ll E$, we deduce from the above that
\begin{align}
\label{secondcase}
\begin{split}
     \sum_{\substack{1\leq l_1, \ldots l_{k-1}\leq K}}  K^{2m-k} & |\mathcal{B}_{K}|^{2m-k-1} \int\limits_{\substack{\mathcal{C}_{K,l_1, \cdots, l_{k-1}} \\ |u| \geq 5}} L(\mathbf{t},u) \dif u \dif \mathbf{t}  \\
     \ll & X(\log X)^{(2m-k)^2/4+1}E^{2m}(\log \log E)^{O(1)}(\log\log X)^{O(1)} \\
     & \hspace*{2cm} \times \sum_{\substack{1\leq l_1, \ldots l_{k-1}\leq K}} \exp\Big( \sum^{k-1}_{i=1}l_i-\sum^{k-1}_{i=1} \sum^{k}_{j=i+1}\max(l_i,\ldots, l_{j-1} )\Big) \\
     \ll &   X(\log X)^{(2m-k)^2/4+1}E^{2m}(\log \log E)^{O(1)}(\log\log X)^{O(1)}.
\end{split}
\end{align}

   As we have $|t_i| \geq 5$ for any $1 \leq i \leq k$ when $|u| \leq 5$, we see that the volume of the region $\mathcal{C}_{K,l_1, \cdots, l_{k-1}}$ when $|u| \leq 5$ is $\ll E^{k}$. It follows that 
\begin{align}
\label{secondcaseusmall}
\begin{split}
 \sum_{\substack{1\leq l_1, \ldots l_{k-1}\leq K}}  K^{2m-k} & |\mathcal{B}_{K}|^{2m-k-1} \int\limits_{\substack{\mathcal{C}_{K,l_1, \cdots, l_{k-1}} \\ |u| \leq 5}} L(\mathbf{t},u) \dif u \dif \mathbf{t}  \\
     \ll & X(\log X)^{2m^2-2mk+3k^2/4+m-3k/4}E^{2m-1}(\log \log E)^{O(1)}(\log\log X)^{O(1)} \\
     & \hspace*{2cm} \times \sum_{\substack{1\leq l_1, \ldots l_{k-1}\leq K}} \exp\Big( -\sum^{k-1}_{i=1} \sum^{k}_{j=i+1}\max(l_i,\ldots, l_{j-1} )\Big) \\
     \ll &   X(\log X)^{2m^2-2mk+3k^2/4+m-3k/4}E^{2m-1}(\log \log E)^{O(1)}(\log\log X)^{O(1)}.
\end{split}
\end{align}

   We now deduce the estimation in \eqref{finiteintest} using \eqref{firstcase}--\eqref{secondcaseusmall}. This completes the proof of the proposition.
\end{proof}

\section{Proof of Lemma \ref{fdiff}}

We apply the Cauchy-Schwarz inequality to see that
\begin{align}
\label{pocs1}
\begin{split}
   \sumstar_{\substack{d \leq X \\ (d,2)=1}} & \bigg|\sum_{n\leq Y} \chi^{(8d)}(n)\big (1-\Phi_U(\frac {n}{Y})\big )\bigg|^{2m} \\
    \leq & \bigg(\sumstar_{\substack{d \leq X \\ (d,2)=1}}\bigg|\sum_{n\leq Y} \chi^{(8d)}(n)\Big (1-\Phi_U \Big( \frac {n}{Y} \Big) \Big )\bigg|^{2}\bigg)^{1/2}
    \bigg(\sumstar_{\substack{d \leq X \\ (d,2)=1}}\bigg|\sum_{n\leq Y} \chi^{(8d)}(n)\Big (1-\Phi_U \Big( \frac {n}{Y} \Big)\Big )\bigg|^{4m-2}\bigg)^{1/2}.
\end{split}
\end{align}
It follows from \cite[Corollary 2]{HB1} that for arbitrary complex numbers $a_n$, for $X, Z \geq 2$ and any
   $\varepsilon>0$, we have
\begin{align*}
\begin{split}
  &  \sumstar_{\substack{d \leq X \\ (d,2)=1}}\bigg|\sum_{n\leq Z} a_n\chi^{(8d)}(n)\bigg|^{2} \ll  (XZ)^{\varepsilon}(X+Z)\sum_{\substack{m,n \leq Z \\ mn=\square}}|a_ma_n|.
\end{split}
\end{align*}

The above with $Z=Y$, mindful of our assumption that $Y \leq X$, leads to
\begin{align*}
\begin{split}
  \sumstar_{\substack{d \leq X \\ (d,2)=1}}\bigg|\sum_{n\leq Y} \chi^{(8d)}(n) \Big (1-\Phi_U \Big( \frac {n}{Y} \Big)\Big ) \bigg|^{2} \ll & (XY)^{\varepsilon}(X+Y)\sum_{\substack{n_1, n_2 \leq Y \\ n_1n_2=\square} }\Big (1-\Phi_U \Big( \frac {n_1}{Y} \Big)\Big )\Big (1-\Phi_U \Big(\frac {n_2}{Y} \Big) \Big ) \\
  \ll &
  X^{1+\varepsilon}\sum_{\substack{Y(1-1/U) \leq n_1, n_2 \leq Y \\ n_1n_2=\square} }1 .
\end{split}
\end{align*}
  We evaluate the last expression above by writing $n_1=d_1m^2_1, n_2=d_1m^2_2$ with $d_1$ square.  The above is
\begin{align}
\label{pocs3}
\begin{split}
  \ll X^{1+\varepsilon}\sum_{\substack{d_1 \leq Y}}\sum_{\substack{(Y(1-1/U)/d_1)^{1/2} \leq m_1, m_2 \leq (Y/d_1)^{1/2}}}1 \ll & X^{1+\varepsilon}\sum_{\substack{d_1 \leq Y}} \Big((Y/d_1)^{1/2}-(Y(1-1/U)/d_1)^{1/2}\Big)^2 \\
   \ll & X^{1+\varepsilon}\sum_{\substack{d_1 \leq Y}} \frac {Y}{d_1U^2} \ll X^{1+\varepsilon}YU^{-2}\log Y,
\end{split}
\end{align}
where the penultimate bound comes by virtue of the mean value theorem. \newline

Next note that
\begin{equation}
\label{pocs4}
\sumstar_{\substack{d \leq X \\ (d,2)=1}}\bigg|\sum_{n\leq Y} \chi^{(8d)}(n)\Big (1-\Phi_U \Big( \frac {n}{Y} \Big)\Big )\bigg|^{4m-2} \ll \sumstar_{\substack{d \leq X \\ (d,2)=1}}\bigg|\sum_{n\leq Y} \chi^{(8d)}(n)\bigg|^{4m-2}+\sumstar_{\substack{d \leq X \\ (d,2)=1}}\bigg|\sum_{n\leq Y} \chi^{(8d)}(n)\Phi_U \Big( \frac {n}{Y} \Big)\bigg|^{4m-2}.
\end{equation}

 We deduce by arguing similar to those from \eqref{charint}--\eqref{Ssimplified} that 
\begin{align*}
\begin{split}
\sumstar_{\substack{d \leq X \\ (d,2)=1}}\bigg|\sum_{n\leq Y} \chi^{(8d)}(n)\Phi_U \Big( \frac {n}{Y} \Big)\bigg|^{4m-2} 
   \ll & Y^{2m-1} \sumstar_{\substack{d \leq X \\ (d,2)=1}}
   \Big | \int\limits_{\substack{ (1/2) \\ |t| \leq X^{\varepsilon}}}\Big |L( \tfrac{1}{2}+it, \chi^{(8d)})\Big |\frac 1{1+|t|} \dif t\Big |^{2m}
   +O(XY^{2m-1}).
\end{split}
\end{align*}
  We then deduce from the above and Lemma \ref{Lsmooth} that 
\begin{equation}
\label{pocs6}
 \sumstar_{\substack{d \leq X \\ (d,2)=1}}\bigg|\sum_{n\leq Y} \chi^{(8d)}(n)\Phi_U \Big( \frac {n}{Y} \Big)\bigg|^{4m-2} \ll XY^{2m-1}(\log X)^{O(1)}.
\end{equation}  

 To estimate the first expression on the right-hand side of \eqref{pocs4}, we apply Perron's formula as given in \cite[Corollary 5.3]{MVa1}.  Hence
\begin{align}
\label{Perron}
\begin{split}
    \sum_{n\leq Y}\chi^{(8d)}(n)= & \frac 1{ 2\pi i}\int\limits_{1+1/\log Y-iY}^{1+1/\log Y+iY}L(s,\chi^{(8d)}) \frac{Y^s}{s} \dif s +R_1+R_2, \\
     = &\frac 1{ 2\pi i}\int\limits_{1/2-iY}^{1/2+iY} +\frac 1{ 2\pi i}\int\limits_{1+1/\log Y -iY}^{1/2-iY}+\frac 1{ 2\pi i}\int\limits_{1/2+iY}^{1+1/\log Y+iY} L(s,\chi^{(8d)})\frac{Y^s}{s}\dif s +R_1+R_2,
\end{split}
\end{align}
  where
\begin{align}
\label{R12}
\begin{split}
  R_1 \ll \sum_{\substack{Y/2< n <2Y  \\  n \neq Y  }}\min \left( 1, \frac {1}{|n-Y|} \right)  \ll \log Y \quad \mbox{and} \quad R_2 \ll \frac {4^{1+1/\log Y}+Y^{1+1/\log Y}}{Y}\zeta(1+1/\log Y)\ll \log Y.
\end{split}
\end{align}
  Here the last estimation above follows from \cite[Corollary 1.17]{MVa1}.
We now consider the moments of the horizontal integrals in \eqref{Perron}. We may assume that $Y\geq 10$, otherwise the lemma is trivial. 
By symmetry we only need to consider only one of them. Note that we have $|Y^s/s| \leq Y^{1+1/\log Y}/Y \ll 1$ in that range and $m \geq 1$, which allows us to apply Hölder's inequality to get
\begin{align}
\label{horizontalint}
\begin{split}
 \sumstar_{\substack{d \leq X \\ (d,2)=1}}\bigg| \int\limits_{1/2+iY}^{1+1/\log Y+iY} L(s,\chi^{(8d)})\frac{Y^s}{s} \dif s\bigg|^{4m-2}\ll & \sumstar_{\substack{d \leq X \\ (d,2)=1}}\bigg( \int\limits_{1/2+iY}^{1+1/\log Y+iY} |L(s,\chi^{(8d)})| |\dif s| \bigg)^{4m-2} \\
 \ll &  \sumstar_{\substack{d \leq X \\ (d,2)=1}}\int\limits_{1/2+iY}^{1+1/\log Y+iY} |L(s,\chi^{(8d)})|^{4m-2} |\dif s| \ll X(\log X)^{O(1)}.
\end{split}
\end{align}
To get the last bound above, we utilize Lemma \ref{prop: upperbound}, which gives that for $ 1/2 \leq \Re(s) \leq 1+1/\log Y$, under GRH, 
$$\sumstar_{\substack{d \leq X \\ (d,2)=1}}|L(s,\chi^{(8d)})|^{4m-2} \ll X(\log X)^{O(1)}.$$

  We treat the vertical integral in \eqref{Perron} using Hölder's inequality, Proposition \ref{t3prop} with $k=1$ and the assumption $Y \leq X$.  Thus
\begin{align}
\label{verticalint}
\begin{split}
  \sumstar_{\substack{d \leq X \\ (d,2)=1}}\bigg|\int\limits_{1/2-iY}^{1/2+iY}L(s,\chi^{(8d)}) \frac{Y^s}{s} \dif s\bigg|^{4m-2} \ll & Y^{2m-1}\sumstar_{\substack{d \leq X \\ (d,2)=1}} \bigg( \int\limits_{0}^Y \frac{|L( \tfrac{1}{2}+it,\chi^{(8d)}) |}{t+1} \dif t \bigg)^{4m-2}  \\
   \ll &  Y^{2m-1}\sum_{n\leq \log Y+2} \frac{n^{4m-2} }{e^{(4m-2)n}} \sumstar_{\substack{d \leq X \\ (d,2)=1}} \bigg( \int\limits_{e^{n-1}-1}^{e^{n}-1 } |L( \tfrac{1}{2}+it,\chi^{(8d)}) | \dif t \bigg)^{4m-2} \\
  \ll & Y^{2m-1}X(\log X)^{O(1)}\Big ( \sum_{n\leq \log Y+2}\frac{n^{4m-2}}{e^{(4m-2)n} }e^n +\sum_{n\leq \log Y+2}n^{4m-2} \Big )\\
  \ll & Y^{2m-1}X(\log X)^{O(1)}.
\end{split}
\end{align}
From \eqref{Perron}--\eqref{verticalint}, we infer
\begin{equation}
\label{pocs10}
   \sumstar_{\substack{d \leq X \\ (d,2)=1}}\bigg|\sum_{n\leq Y} \chi^{(8d)}(n)\bigg|^{4m-2}\ll Y^{2m-1}X(\log X)^{O(1)}.
\end{equation}
Finally, from \eqref{pocs1},  \eqref{pocs3}, \eqref{pocs6}, \eqref{pocs10} and $U=X^{2\varepsilon}, Y \leq X$, \eqref{theorem3firstrest} is valid.  This completes the proof of the lemma.

\vspace*{.5cm}

\noindent{\bf Acknowledgments.}  P. G. is supported in part by NSFC grant 12471003 and L. Z. by the FRG Grant PS71536 at the University of New South Wales.  The authors would like to thank the anonymous referee for his/her very careful inspection of the paper and many valuable suggestions.

\bibliography{biblio}
\bibliographystyle{amsxport}

\end{document}